\documentclass[11pt,twoside]{article}
\usepackage[margin=1in]{geometry}
\usepackage{graphicx}        
\usepackage{multicol}        
\usepackage[bottom]{footmisc}

\usepackage{newtxtext}       %

\usepackage{layout}

\usepackage{url}
\usepackage{amsmath}
\usepackage{amssymb}
\usepackage{amsfonts}
\usepackage{bm}

\usepackage{graphicx}
\usepackage{float}
\usepackage[justification=centering]{caption}
\usepackage{tikz}
\usepackage[ruled,vlined]{algorithm2e}
\usepackage{algorithmic}
\usepackage{wrapfig}
\usepackage{mathtools}
\usepackage{xcolor}
\usepackage{amsthm}
\usepackage{cite}

\usepackage{fancyhdr} 
\usetikzlibrary{calc,intersections,matrix,chains,positioning,decorations.pathreplacing,arrows}

\newtheorem{thm}{Theorem}[section]

\newtheorem{rmk}[thm]{Remark}
\newtheorem{lemma}[thm]{Lemma}

\newcommand{\R}{\mathbb{R}}
\newcommand{\C}{\mathbb{C}}
\newcommand{\N}{\mathbb{N}}

\newcommand{\mx}{\bm x}

\newcommand{\malpha}{\bm \alpha}
\newcommand{\mj}{\bm j}
\newcommand{\mi}{\bm i}

\newcommand{\my}{\bm y}
\newcommand{\mN}{\bm N}

\newcommand{\0}{\bm 0}
\newcommand{\1}{\bm 1}
\newcommand{\bxi}{\boldsymbol{\xi}}

\newcommand{\mf}{\bm f}
\newcommand{\ihat}{\hat \imath}
\newcommand{\supp}{\text{\rm supp}}

\newcommand{\sinc}{\text{sinc}}

\newcommand{\bbf}{\bm f}
\newcommand{\bbg}{\bm g}
\newcommand{\bbI}{\bm I}

\newcommand{\tA}{\boldsymbol{\mathcal{A}}}
\newcommand{\tB}{\boldsymbol{\mathcal{B}}}
\newcommand{\tC}{\boldsymbol{\mathcal{C}}}
\newcommand{\tV}{\boldsymbol{\mathcal{V}}}
\newcommand{\tW}{\boldsymbol{\mathcal{W}}}
\newcommand{\tF}{\boldsymbol{\mathcal{F}}}
\newcommand{\tG}{\boldsymbol{\mathcal{G}}}
\newcommand{\tI}{\boldsymbol{\mathcal{I}}}

\newcommand{\matA}{\bm A}
\newcommand{\matB}{\bm B}
\newcommand{\matU}{\bm U}
\newcommand{\matSi}{\bm \Sigma}
\newcommand{\matV}{\bm V}
\newcommand{\matE}{\bm E}
\newcommand{\matQ}{\bm Q}
\newcommand{\matY}{\bm Y}
\newcommand{\matR}{\bm R}

\newcommand{\matOm}{\bm \Omega}

\newcommand{\vv}{\bm v}

\newcommand{\tref}[1]{Table~\ref{#1}}
\newcommand{\aref}[1]{Algorithm~\ref{#1}}
\newcommand{\sref}[1]{Section~\ref{#1}}

\allowdisplaybreaks[1]

\numberwithin{equation}{section}
\numberwithin{figure}{section}
\numberwithin{table}{section}

\title{Denoising Convolution Algorithms and Applications \\to SAR Signal Processing}
\author{Alina Chertock\thanks{Department of Mathematics, North Carolina State University, Raleigh, NC, USA; {\tt chertock@math.ncsu.edu}},~
Chris Leonard\thanks{Department of Mathematics, North Carolina State University, Raleigh, NC, USA; {\tt cleonar@ncsu.edu}},~
Semyon Tsynkov\thanks{Corresponding author. Department of Mathematics, North Carolina State University, Raleigh, NC, USA; {\tt  tsynkov@math.ncsu.edu}},~
Sergey Utyuzhnikov\thanks{Department of Mechanical, Aerospace \& Civil Engineering, University of Manchester, Manchester, UK; {\tt S.Utyuzhnikov@manchester.ac.uk}}}
\date{\today}

\begin{document}

\maketitle

\begin{abstract}
Convolutions are one of the most important operations in signal processing. They often involve large arrays and require significant
computing time. Moreover, in practice, the signal data to be processed by convolution may be corrupted by noise. In this paper, we
introduce a new method for computing the convolutions in the quantized tensor train (QTT) format and removing noise from data using the QTT
decomposition. We demonstrate the performance of our method using a common mathematical model for synthetic aperture radar (SAR) processing
that involves a sinc kernel and present the entire cost of decomposing the original data array, computing the convolutions, and then
reformatting the data back into full arrays.
\end{abstract}

\section{Introduction}

Convolution operations are used in different practical applications. They often involve large arrays of data and require optimization with
respect to memory and computational cost. While input data are usually available only in a discrete form, the standard realization
based on a vector-matrix representation is not often efficient since it leads to using sparse matrices. On the
other hand, a tensor decomposition looks very attractive because it might reduce the volume of data very drastically, minimizing the number of zero elements. In
addition, arithmetic operations between tensors can be implemented efficiently.

There are different forms of tensor decomposition. The most popular approach is based on the canonical decomposition \cite{H70} where a
multidimensional array is represented (might be approximate) via a sum of outer products of vectors.  For matrices, such decomposition
is reduced to skeleton decomposition. However, it is known to be unstable in the cases of multiple tensor dimensions, also referred to as
tensor modes. The Tucker decomposition \cite{T66} represents a natural stable generalization of the canonical decomposition and can provide
a high compression rate. The main drawback of the Tucker decomposition is related to the so-called curse of dimensionality; that is, the
algorithm's complexity grows exponentially with the number of tensor modes.
A way to overcome these difficulties is to use the Tensor Train (TT) decomposition, which was originally introduced in \cite{OT09,OT10}.
Effectively, the TT decomposition represents a generalization of the classical SVD decomposition to the case of multiple modes. It can also
be interpreted as a hierarchical Tucker decomposition \cite{H09}.

Computing the TT decomposition fully can be very expensive if we use the standard TT-SVD algorithm given, e.g., by \aref{alg:tt_svd} below.
Therefore, many modifications to this algorithm were proposed in the literature to help speed it up. One such improvement was presented in
\cite{LYB22}, where results comparable to those obtained by the TT-SVD algorithm were produced in a fraction of
the time for sparse tensor data. Another algorithm that uses the column space of the unfolding tensors was designed to compute the TT cores in parallel; see
\cite{SRT21}.
The most popular approach to efficiently compute the TT decomposition is based on using a randomized algorithm; see, e.g.,
\cite{HSW18,CW19,FZ19}.

Maximal compression with the TT decomposition can be reached with matrices whose dimensions are powers of two, as proposed in
the so-called Quantized TT (QTT) algorithm \cite{O10}. As shown in \cite{KKT13}, the convolution realized for multilevel Toeplitz matrices via
QTT has a logarithmic complexity with respect to the number of elements in each mode, $N$, and is proportional to the number of modes. It is
proven that the result cannot be asymptotically improved. However, this algorithm is improved for finite and practically important $N\sim 10^4$
in \cite{RO15} thanks to the cross-convolution in the Fourier (image) space. The improvement is demonstrated for convolutions with three modes
with Newton's potential. It is to be noted that QTT can also be applied to the Fast Fourier Transform (FTT) to decrease its complexity, as shown in \cite{DKS12}. This super-fast FFT (QTT-FFT) beats the standard FFT for extremely large $N$ such as $N\sim 2^{60}$ for one mode
tensors and $N\sim 2^{20}$ for tensors with three modes.

For practical applications, a critical issue is denoising. Real-life data, such as radar signals, are typically contaminated with noise.
Denoising is not addressed in the papers we have cited previously. However, TT decomposition itself potentially has the property of denoising,
owing to the SVD  incorporated in the algorithm \cite{EK19,GCCA20}. In the current work, we propose and implement the low-rank modifications
for the previously developed TT-SVD algorithm of \cite{O11}. These modifications speed up the computations. We also demonstrate the denoising
capacity of numerical convolutions computed using the QTT decomposition. Specifically, we employ a common model for synthetic aperture radar
(SAR) signal processing based on the convolution with a sinc imaging kernel (called the generalized ambiguity function)
\cite[Chapter 2]{sarbook} and show that when a convolution with this kernel is evaluated in the QTT format, the noise level in the resulting
image is substantially reduced compared to that in the original data.

It should be observed that most papers on tensor convolution only consider the run time cost of the convolution after the tensor decomposition
has been applied to the objective function and the kernel function and either ignores the cost of the actual tensor decompositions or puts it
as a side note. In this paper, we consider every step of computing the convolution using the QTT-FFT algorithm, including the decomposition of
the arrays into the QTT format using the TT-SVD algorithm (see \aref{alg:tt_svd}), computation of the QTT-FFT algorithm once in that format,
and then extracting the data back after the computation is conducted (\sref{sec:computing}). As the QTT decomposition is computationally
expensive, we consider several approaches to speed up the decomposition run time. Without these modifications to the TT-SVD decomposition
algorithm, the convolution can take a long to compute and is not a practical approach. We provide more detail in section
\ref{ssec:qtt_e1}.

The methods we use to speed up our TT decompositions are based on truncating SVD ranks in the decomposition algorithm (\aref{alg:tt_svd}) and
lead to a significant noise reduction in the data (Section~\ref{sec:denoising}). Thus, in Chapter 6, we present algorithms to compute
convolutions in a reasonable time while significantly reducing the noise in the data at the same time. Our contribution includes developing and analyzing new approaches to speeding up the tensor train decomposition, see
\sref{sec:computing}. In \sref{sec:denoising}, we consider the effects of convolutions on removing noise in data. Finally, in
\sref{sec:ex_conv}, we show numerical examples and compare our results with other approaches to computing convolutions.

\section{Convolution}\label{sec:convolution}
The convolution operation is widely used in different applications in signal processing, data imaging, physics, and probability, to name a
few. This operation is a way to combine two signals, usually represented as functions, and produce a third signal with meaningful information.
The $D$-dimensional convolution between two functions $f$ and $g$ is defined as
\begin{equation}\label{conv}
I(\mx) =[f*g](\mx) = \int_{\R^D} f(\my)g(\mx-\my)\; d\my,\quad \forall \mx\in\R^D.
\end{equation}

Often to compute the convolution numerically, we assume the support of $f$ and $g$, denoted $\supp(f)$ and $\supp(g)$ respectively, are
compact. For simplicity, in this paper, we assume $\supp(f)=\supp(g)=[-L,L]^D$ for some $L\in\R$. Next, we discretize the domain $[-L,L]^D$
uniformly into $N^D$ points such that
\begin{gather*}\label{space_disc}
\mx_{\mj} = (x_{j_1} ,\ldots,x_{j_D} ),\\
x_{j_d}= -L+\frac{\Delta x}{2}+j_d \Delta x, \quad j_d=0,\ldots,N-1,\quad d=1,\ldots,D,
\end{gather*}
where $\Delta x =\frac{2L}{N}$ and $\mj=(j_1,\ldots,j_D)$. We then let $\bbf$ and $\bbg$ be $D$-dimensional arrays such that
\begin{equation*}
\bbf_{\mj}=f(\mx_{\mj}),\quad \bbg_{\mj}=g(\mx_{\mj})
\end{equation*}
for all $\mj$.
This leads to the discrete convolution $\bbI$ such that
\begin{equation}\label{dconv}
\bbI_{\mj}\coloneqq(\Delta x)^D\sum_{\mi} \bbf_{\mi}
\bbg_{\mj-\mi+(\frac{N}{2}-1)\1}\approx I(\mx_{\mj}),
\end{equation}
where $\1 = (1,\ldots,1)$ and the sums are over all indices $\mi=(i_1,\ldots,i_D)$ that lead to legal subscripts. This Riemann sum
approximation \eqref{dconv} to the integral \eqref{conv} uses the midpoint rule, thus having $\mathcal{O}(\Delta x^2)$ accuracy.

\begin{rmk}
The convolution defined in \eqref{dconv} is equivalent to Matlab's \textbf{convn} function with the optional shape input set to 'same' and then
multiplied by $(\Delta x)^D$.
\end{rmk}

To compute this convolution directly takes $\mathcal{O}(N^{2D})$ operations, but it can be reduced to $\mathcal{O}(N^D\log(N^D))$ by using the
fast Fourier transform (FFT) and the discrete convolution theorem. The FFT algorithm is an efficient algorithm used to compute the
$D$-dimensional discrete Fourier transforms (DFT) of $\tV\in\R^{N\times\ldots\times N}$,
\begin{equation*}
\hat{\tV}_{\malpha}\coloneqq DFT(\tV)=
\sum_{ \mj=\0}^{\mN-\1}\tV_{\mj}\omega_{N}^{\mj\cdot\malpha}
\end{equation*}
where the sum is over the multi-indexed array $\mj$,
\begin{gather*}
\malpha = (\alpha_1,\ldots,\alpha_D), \quad \alpha_d=0,\ldots,N-1, \quad d=1,\ldots,D,\\
\mN = (N,\ldots,N),\quad \0=(0,\ldots,0),
\end{gather*}
and $\omega_{N}=e^{-\frac{2\pi \ihat}{N}}$, where $\ihat=\sqrt{-1}$ is the imaginary unit.
Similarly, the $D$-dimensional inverse discrete Fourier transform (IDFT), such that
\begin{equation*}
\tV=IDFT(DFT(\tV)),
\end{equation*}
of the array $\hat{\tV}\in\R^{N\times\ldots\times N}$ is given by
\begin{equation*}
\tV_{\mj} =\frac{1}{N^D} \sum_{\malpha=\0}^{\mN-\1}\hat{\tV}_{\malpha}\omega_{N}^{-\mj\cdot\malpha}.
\end{equation*}

Using the discrete Fourier transform, we can compute the circular convolution $\bbI^c=(\tV\circledast \tW)$ defined as
\begin{gather*}
\bbI^c_{\mj} = \sum_{\mi=\0}^{\mN-\1}\tV_{\mi}\bar{\tW}_{\mj-\mi}\\
\bar{\tW}_{i_1,\ldots,i_D} = \tW_{j_1,\ldots,j_D},
\quad i_d\equiv j_d \text{ mod($N$)}, \quad d=1,\ldots,D,
\end{gather*}
by taking the DFT of $\tW$ and $\tV$, multiplying the results together, and then taking the IDFT of the given result. Thus, we have
\begin{equation*}
\bbI^c =IDFT(DFT(\tW)\odot DFT(\tV))
\end{equation*}
where $\odot$ is Hadamard product (element-wise product) of $D$-dimensional arrays.
The circular convolution is the same as the convolution of two periodic functions (up to a constant scaling),
thus to obtain the  convolution given
in \eqref{dconv} (also known as a linear convolution), we need to pad the vectors $\bbf$ and $\bbg$ with at least $N-1$ zeros in each
dimension. For example, given the vectors $\bbf^0, \bbg^0 \in\R^{2N-1}$ with
$$
\begin{aligned}
\bbf^0_j=\begin{cases}
\bbf_j & 0\leq j\leq N-1\\
0 & j>N-1
\end{cases}, \quad\text{ and }\quad
\bbg^0_j=\begin{cases}
\bbg_j & 0\leq j\leq N-1\\
0 & j>N-1
\end{cases},
\end{aligned}
$$
and $\bbI^c=(\bbf^0\circledast \bbg^0)$ as the circular convolution between them, the linear convolution $\bbI$ in \eqref{dconv} is given by
\begin{equation*}
    \bbI_j = \Delta x
    \bbI^c_{j+\frac{N-1}{2}},\quad j=0,\ldots,N-1.
\end{equation*}

In this paper, we let $g$ be a predefined kernel, such as the SAR generalized ambiguity function (GAF) (see \sref{sec:SAR} and
\cite[Chapter 2]{sarbook} for detail) and $f$ be a smooth gradually varying function contaminated with white noise. To compute the convolution,
we use the QTT decomposition \cite{K11} and the QTT-FFT algorithm \cite{DKS12}. The QTT decomposition is a particular case of the more general
TT decomposition (see \sref{sec:TTD} and \cite{O11} for detail).

\section{Synthetic aperture radar (SAR)}\label{sec:SAR}
SAR is a coherent remote sensing technology capable of producing two-dimensional images of the Earth's surface from overhead platforms
(airborne or spaceborne). SAR illuminates the chosen area on the surface of the Earth with microwaves (specially modulated pulses) and
generates the image by digitally processing the returns (i.e., reflected signals). SAR processing involves the application of the matched
filter and summation along the synthetic array, which is a collection of successive locations of the SAR antenna along the flight path. Matched
filtering yields the image in the direction normal to the platform flight trajectory or orbit (called cross-track or range), while summation
along the array yields the image in the direction parallel to the trajectory or orbit (along-the-track or azimuth).

Mathematically, each of the two signal processing stages can be interpreted as the convolution of the signal received by the SAR antenna with a
known function. Equivalently, it can be represented as a convolution of the ground reflectivity function, which is the unknown quantity that
SAR aims to reconstruct the imaging kernel or generalized ambiguity function. The advantage of this equivalent representation is that it
leads to a very convenient partition: the GAF depends on the imaging system's characteristics, whereas the target's properties determine the
ground reflectivity function. Moreover, image representation via GAF allows one to see clearly how signal compression (a property that pertains
to SAR interrogating waveforms) enables SAR resolution, i.e., the capacity of the sensor to distinguish between closely located targets.

In the simplest possible imaging scenario, when the propagation of radar signals between the antenna and the target is assumed unobstructed,
and several additional assumptions also hold; the GAF in either range or azimuthal direction is given by the sinc (or spherical Bessel)
function:
\begin{equation}
\label{eq:GAF}
g(x)=A\,\sinc\left(\pi\frac{x}{\Delta_x}\right)\equiv A\frac{\sin\left(\pi\frac{x}{\Delta_x}\right)}{\pi\frac{x}{\Delta_x}},
\end{equation}
where the constant $A$ is determined by normalization, $x$ denotes a given direction, and the quantity $\Delta_x$ is the resolution in this
direction. From the formula (\ref{eq:GAF}), we see that the resolution is defined as half-width of the sinc main lobe, i.e., the distance from
is central maximum to the first zero. When $x$ is the range direction (cross-track), the resolution $\Delta_x$ is inversely proportional to the
SAR signal bandwidth, see \cite[Section 2.4.4]{sarbook}. When $x$ is the azimuthal direction (along-the-track), the resolution is inversely
proportional to the length of the synthetic array, i.e., synthetic aperture, see \cite[Section 2.4.3]{sarbook}. Note that lower values of $
\Delta_x$ correspond to better resolution because SAR can tell between the targets located closer to one another. It can also be shown that as
$\Delta_x\to0$ the GAF given by (\ref{eq:GAF}) converges to the $\delta$-function in the sense of distributions \cite[Section 3.3]{a92e}. In
this case, the image, which is a convolution of the ground reflectivity with the GAF, coincides with ground reflectivity. This would be ideal
because the image would reconstruct the unknown ground reflectivity exactly. This situation, however, is never realized in practice because
having $\Delta_x\to0$ requires either the SAR bandwidth (range direction) or synthetic aperture (azimuthal direction) to become infinitely
large, which is not possible.

The literature on SAR imaging is vast. Among the more mathematical sources, we mention the monographs \cite{cheney-09}, and \cite{sarbook}.

\section{Tensor Train Decomposition}\label{sec:TTD}
Consider the $K$-mode, tensor $\tA\in\C^{M_1\times\ldots\times M_K}$ such that
\begin{equation*}
\tA=a(i_1,\ldots,i_K), \quad i_k=0,\ldots,M_k-1, \quad k=1,\ldots,K,
\end{equation*}
where $M_k$ is the size of each mode, and $a(i_1,\ldots,i_K)\in\C$ are the elements of the tensor $\tA$ for all $i_k=0,\ldots,M_k-1$ and
$k=1,\ldots,K$. The tensor train format of $\tA$ decomposes the tensor into $K$ cores $\tA^{(k)}\in\C^{r_{k-1}\times M_k\times r_k}$ such that
\begin{equation*}
a(i_1,\ldots,i_K)= \matA^{(1)}_{i_1}\matA^{(2)}_{i_2}\cdots \matA^{(K)}_{i_K},
\end{equation*}
where the matrices  $\tA^{(k)}(:,i_k,:)=\matA^{(k)}_{i_k}\in\C^{r_{k-1}\times r_k},$ for all $i_k=0,\ldots,M_k-1$, $k=1,\ldots,K$ (In Matlab
notation, $\matA^{(k)}_{i_k}=\text{squeeze}(\tA^{(k)}(:,i_k,:))$, where squeez() is used to convert the $\C^{r_{k-1}\times 1\times r_k}$ tensor
into a $\C^{r_{k-1}\times r_k}$ matrix). The matrix dimensions $r_k$, $k=1,\ldots,K$, are referred to as the TT-ranks of the tensor
decomposition, and the $3-$mode tensors $\tA^{(k)}$ are the TT-cores. Since we are interested in the case when $a(i_1,\ldots,i_K)\in\C$,
we impose the condition $r_0=r_K=1$. Let $M=\max_{1\leq k\leq K} M_k$ and $r=\max_{1\leq k\leq K-1}r_k$, then the the tensor $\tA$, which has
$\mathcal{O}(M^K)$ elements,  can be represented with $\mathcal{O}(MK r^2)$ elements in the TT format.

We can also represent the TT decomposition as the product of tensor contraction operators. Define the tensor contraction between the tensors
$\tA\in\C^{M_1\times\ldots\times M_K}$ and $\tB\in\C^{M_K\times\ldots\times M_{\tilde{K}}}$ (note that the first dimension size of $\tB$ equals
the last dimension size of $\tA$) as $\tC = \tA\circ\tB\in\C^{M_1\times\ldots\times M_{K-1}\times M_{K+1}\times\ldots\times M_{\tilde{K}}}$
where
\begin{equation*}
\tC(i_1,\ldots,i_{K-1},i_{K+1},\ldots,i_{\tilde{K}}) = \sum_{p=0}^{M_K-1} \tA(i_1,\ldots,p)\tB(p,\ldots,i_{\tilde{K}}).
\end{equation*}
Then the TT format of $\tA$ can be represented as
\begin{equation*}
\tA = \tA^{(1)}\circ\ldots\circ\tA^{(K)}.
\end{equation*}

Before we show how to find the TT-cores, we first need to define a few properties of tensors. First, let the matrix $\matA^{\{k\}}$ be the
$k$-th unfolding of the tensor $\tA$ such that
\begin{align*}
\matA^{\{k\}}(\alpha,\beta) &= a(i_1,\ldots,i_K),\\
\alpha = i_1+i_2M_1+\ldots+i_{k}\Pi_{l=1}^{k-1}M_l, &\quad \beta = i_{k+1}+i_{k+2}M_{k+1}+\ldots+i_{K}\Pi_{l=k+1}^{K-1}M_l.
\end{align*}
Thus, we have that $\matA^{\{k\}}\in\C^{M_1M_2\ldots M_k\times M_{k+1}M_{k+2}\ldots M_K}$ which we write as
\begin{equation*}
\matA^{\{k\}} = a(i_1\ldots i_k,i_{k+1}\ldots i_K).
\end{equation*}
We denote the process of unfolding a tensor $\tA$ into a matrix $\matA^{\{k\}}\in\C^{M_1M_2\ldots M_k\times M_{k+1}M_{k+2}\ldots M_K}$ as
\begin{equation*}
\matA^{\{k\}}=\text{reshape}(\tA,[M_1M_2\ldots M_k, M_{k+1}M_{k+2}\ldots M_K])
\end{equation*}
and folding a matrix into a tensor $\tA\in\C^{M_1\times\ldots \times M_K}$ as
\begin{equation*}
\tA=\text{reshape}(\matA^{\{k\}},[M_1,M_2,\ldots,M_K]).
\end{equation*}
(Note this is to be consistent with the Matlab function reshape()).

From \cite{O11} it can be shown that there exist a TT-decomposition of $\tA$ such that
\begin{equation*}
r_k=\text{rank}(A^{\{k\}}), \quad k=1,\ldots,K.
\end{equation*}
Denote the Frobenius norm of a tensor $\tA\in\C^{M_1\times\ldots\times M_K}$ as
\begin{equation*}
\|\tA\|_F=\sqrt{\sum_{i_1=0}^{M_1-1}\ldots\sum_{i_K=0}^{M_K-1}|a(i_1,\ldots,i_K)|^2},
\end{equation*}
and the $\varepsilon_k$-rank of the matrix $\matA^{\{k\}}$ as
\begin{equation*}
    \text{rank}_{\varepsilon_k}(\matA^{\{k\}})\coloneqq \min \{\text{rank}(\matB) : \|\matA^{\{k\}}-\matB\|_F\leq \varepsilon_k\}.
\end{equation*}
Given a set $\{\varepsilon_k\}_{k=1}^K$, we can approximate the tensor $\tA$ with a tensor $\tilde{\tA}$ in the TT format such that it has TT-
ranks $\tilde{r}_k\leq\text{ rank}_{\varepsilon_k}(\matA^{\{k\}})$ and
\begin{equation*}
	\|\tA-\tilde{\tA}\|_F \leq \varepsilon,\quad \varepsilon^2=\varepsilon_1^2+\ldots+\varepsilon_{K-1}^2.
\end{equation*}

In \aref{alg:tt_svd}, we present the TT-SVD algorithm \cite{O11}, which computes a TT-decomposition of a tensor $\tA$ with a prescribed
accuracy $\varepsilon$. In
\sref{sec:computing}, we present some modifications to this algorithm that relax the prescribed tolerance and allow us to compute an
approximate decomposition faster. For a tensor $\tA\in\C^{M_1\times\ldots\times M_K}$, define
\begin{equation*}
    |\tA| = \text{number of elements in }\tA= M_1 M_2\ldots M_K.
\end{equation*}
\begin{algorithm}[ht!]
    \footnotesize
    \SetKwInOut{Input}{input}
    \SetKwInOut{Output}{output\,}
    \SetAlgoLined
    \Input{$\tA$, $\varepsilon$}
    \Output{TT-Cores: $\tA^{(1)},\tA^{(2)},...,\tA^{(K)}$}
    $\tau \coloneqq \frac{\varepsilon}{\sqrt{M-1}}\|\tA\|_F$\\
     $r_0\coloneqq1$;\\
     \For{k=1,...,K-1}{
      $\matA^{\{k\}} \coloneqq \text{reshape}(\tA,[M_kr_{k-1},\frac{|\tA|}{M_kr_{k-1}}])$\\
      Compute truncated SVD: $\matU\matSi \matV^*+\matE=\matA^{\{k\}}$ such that $\|\matE\|_F \leq \tau$\\ $r_k\coloneqq\text{rank}(\matSi)=\text{rank}_\tau(\matA^{\{k\}})$\\
      $\tA^{(k)}\coloneqq \text{reshape}(\matU,[r_{k-1}, M_k, r_{k}])$\\
      $\tA \coloneqq \matSi \matV^*$\\
     }
     $\tA^{(K)}\coloneqq \tA$
     \caption{TT-SVD}
     \label{alg:tt_svd}
\end{algorithm}

The TT-decomposition can also be applied to tensors with a small number of modes by using the quantized tensor train decomposition (QTT). For
instance, let $\vv\in\C^{2^K}$ be a vector ($1$-mode tensor). To apply the QTT-decomposition of $\vv$, we reshape it into the $K$-mode tensor
$\tV\in\C^{2\times\ldots\times 2}$ such that
\begin{equation*}\label{vec2qtt}
\tV(i_1,i_2,\ldots,i_K)=\vv(i),
\end{equation*}
where
\begin{equation*}
i =\sum_{k=1}^K i_k2^{k-1},\quad i_k=0,1,
\end{equation*}
then compute the TT-decomposition of the tensor $\tV$ (you can think of $i_K\ldots i_1$ as the binary representation of $i$). Extending the
QTT-decomposition to matrices ($2$-mode tensors) $\matV\in\C^{2^K\times 2^K}$ can be done similarly by reshaping them into  $2K$-mode tensors
$\tV\in\C^{2\times\ldots\times 2}$, then computing the TT-decomposition of $\tV$.

We can approximate the discrete Fourier transform of a vector $\vv\in\R^{2^K}$ (or 2D discrete Fourier transform of a matrix
$\matV\in\R^{2^K\times 2^K}$) in the QTT format using what is known as the QTT-FFT approximation algorithm \cite{DKS12}.
Let $\hat{\vv}=DFT(\vv)$ be the discrete Fourier transform of $\vv$ and
let $\tV$ and $\hat{\tV}$ be the tensors in the QTT-format that represent the vectors $\vv$ and $\hat{\vv}$ respectively. Given $\tV$, the QTT-
FFT approximation algorithm can approximate $\hat{\tV}$ with a tensor $\tilde{\tV}$ such that
\begin{equation}
\label{eq:tolerance}
    \|\tilde{\tV}-\hat{\tV}\|_F\leq \varepsilon
\end{equation}
for some given tolerance $\varepsilon$.
Similarly, we could prescribe some maximum TT-rank, $\hat{R}_{\max}$, for the QTT-FFT algorithm such that $\tilde{r}_k\leq \hat{R}_{\max}$ for
all TT-ranks of $\tilde{\tV}$, $\{\tilde{r}_k\}_{k=0}^K$. The QTT-FFT algorithm can easily be modified to the inverse Fourier transform of a
vector (or matrix) in the QTT format, which we denote as the QTT-iFFT algorithm.

\section{Computing the convolution with QTT decomposition}
\label{sec:computing}
In practice, we often need to compute the convolution \eqref{conv}, where $f$ is the function of interest and $g$ is a given kernel, but $f$ is
not given explicitly. Instead, we are given noisy data
\begin{equation}
\label{eq:noisydata}
    (\bbf_{\xi})_{\mj}=f(\mx_{\mj})+\xi_{\mj}
\end{equation}
at discrete points $\mx_{\mj}$, ${\mj}=(j_1,\ldots,j_D)$. In particular, representing the ground reflectivity function for SAR reconstruction
in the form (\ref{eq:noisydata}) helps one model the noise in the received data. We assume that $\xi_{\mj}$ is white noise from a normal
distribution with the standard deviation $\sigma$, i.e., $\xi_{\mj}\sim\mathcal{N}(0,\sigma^2)$.

Since the kernel function $g$ is known, we can discretize it as
\begin{equation*}
    \bbg_{\mj}=g(\mx_{\mj}),
\end{equation*}
for the same $\mx_{\mj}$ values as in \eqref{eq:noisydata}. We assume the $D$-dimensional spatial domain is uniformly discretized into $N^D$
points where $N=2^{K-1}-1$, see \eqref{space_disc}.
To compute the discrete convolution \eqref{dconv}, we propose using the quantized tensor train (QTT) decomposition. To represent the arrays in
the QTT format, we pad them with zeros such that the new arrays are $D$-mode tensors in $\R^{2^{K}\times\ldots\times 2^{K}}$. We can relax the
condition on the size $N$, but to compute the convolution with an FFT algorithm, we need to zero-pad each dimension with at least $N-1$ extra
zeros (see \sref{sec:convolution}). Also, for the QTT decomposition, we need each dimension to be of size $2^K$ for some $K\in\N$. Let
$\tF_\xi,\tG$ be the zero-padded tensors representing $\bbf_\xi$ and $\bbg$ respectively in the QTT format. Here, we assume that the
discretization of $f$, $\bbf$, has a low, but not exactly known, TT-rank in the QTT-format. This is motivated by the fact that many standard
piecewise smooth functions naturally have a low TT-rank, see \cite{O13,G10,K11}.

To find approximations of these tensors in the TT-format, we modify the original TT-SVD algorithm. This is because with the full TT-SVD
algorithm, if the tolerance $\varepsilon$ is small, see equation (\ref{eq:tolerance}), the TT-decomposition has close to full rank. Not only
does it take a very long time to compute these decompositions, but most of the noise is still present. However, if $\varepsilon$ is too large,
the TT-SVD algorithm loses too much information about the true function $f$. For these reasons, we present slight modifications to the TT-SVD
algorithm. They are needed to significantly reduce the computing time, as illustrated by the example in Section~\ref{ssec:qtt_e1}.

We consider three different modifications to the TT-SVD algorithm. These modifications are as follows:
\begin{enumerate}
\item[(1)] Set some max rank $R_{\max}$ and truncate the SVD in \aref{alg:tt_svd} with ranks less than or equal to this threshold. Denote this
method as the \textbf{max rank TT-SVD} algorithm.
\item[(2)] Set some max rank $R_{\max}$ and replace the SVD in \aref{alg:tt_svd} with a randomized SVD (RSVD) given in \cite{HMT11} with max
ranks set to $R_{\max}$ (see Appendix~\ref{sec:randomized}). Denote this method as the \textbf{max rank TT-RSVD} algorithm. Note that for this
algorithm, we also need to prescribe an oversampling parameter $p$.  We could choose from several randomized SVD algorithms, but due to
simplicity and effectiveness, we use the approach described in  Appendix~\ref{sec:randomized}. This algorithm implements the direct SVD.
\item[(3)] Truncate the SVD in \aref{alg:tt_svd} based on when there is a relative drop in singular values, i.e., if
$\frac{\sigma_{k+1}}{\sigma_k}<\delta~~(0 < \delta < 1)$ for a given threshold $\delta$, then truncate the singular values less than $\sigma_k$. Denote this method
as the \textbf{SV drop off TT-SVD} algorithm.
\end{enumerate}
For the \textbf{max rank TT-RSVD}, if the unfolding matrices $\matA^{\{k\}}\in\R^{m_k\times n_k}$, where $\min(m_k,n_k)\leq R_{\max}+p$, then
we revert to the \textbf{max rank TT-SVD} algorithm (without the randomized SVD).

We can modify the QTT-FFT and QTT-iFFT algorithms similarly to our modifications of the TT-SVD algorithms to get a low-rank approximation to
the discrete Fourier transform representations of $\tF_\xi$ and $\tG$. For this, we replace the SVD in the QTT-FFT algorithm (QTT-iFFT) with
the truncated SVD algorithms (1)-(3) given above, but with possibly a different max rank which we denote $\hat{R}_{max}$ for (1) and (2),
or different threshold $\hat{\delta}$ for (3). For the examples in \sref{sec:ex_conv}, we distinguish between
$R_{max}$ and $\hat{R}_{max}$. However, we use the same threshold for $\delta$ in the TT-SVD algorithm and the QTT-FFT algorithm. Thus, we
do not distinguish between the two. Note that using the threshold (1) in the QTT-FFT algorithm is not new and is mentioned in \cite{DKS12}.

With these above modifications to the TT-SVD algorithm and QTT-FFT (QTT-iFFT) algorithms, we propose the following algorithm
(\aref{alg:qtt_conv}) to approximate the convolution between the $D$-dimensional arrays $\bbf$ and $\bbg$. For this algorithm, we denote
\begin{itemize}
\item $QFFT_{\hat{R}_{\max}(\hat{\delta})}$: QTT-FFT algorithm with a max rank of ${\hat{R}_{\max}}$ (or threshold $\hat{\delta}$),
\item  $QiFFT_{\hat{R}_{\max}(\hat{\delta})}$: QTT-iFFT algorithm with a max rank of ${\hat{R}_{\max}}$ (or threshold $\hat{\delta}$).
\end{itemize}

\begin{algorithm}[ht!]
    \footnotesize
    \SetKwInOut{Input}{input}
    \SetKwInOut{Output}{output\,}
    \SetAlgoLined
    \Input{$\bbf_\xi$, $\bbg$}
    \Output{ $\bbI$ }
    \textbf{Step 1:}
     $\tF_\xi$ = $\text{reshape}(\bbf_\xi,[2,\ldots,2])$,
     $\tG$ = $\text{reshape}(\bbg,[2,\ldots,2])$\\
     \textbf{Step 2:} Decompose $\tF_\xi$ and $\tG$ into the QTT format using one of the modified TT-SVD algorithms.\\
     \textbf{Step 3:} $\tI=QiFFT_{\hat{R}_{max}(\hat{\delta})}(QFFT_{\hat{R}_{max}(\hat{\delta})}(\tF_\xi)\odot QFFT_{\hat{R}_{max}(\hat{\delta})}(\tG))$.\\
     \textbf{Step 4:} Retrieve $\bbI$ from $\tI$. (see \aref{alg:full})
     \caption{QTT convolution}
     \label{alg:qtt_conv}
    \end{algorithm}

In Theorem \ref{thm:run_time}, we show the asymptotic run time behavior of computing a convolution in one spatial dimension ($D=1$) with the
\textbf{max rank TT-SVD} algorithm. First, we prove an auxiliary result about the size of the unfolding matrices for this algorithm; see Lemma
\ref{lemma:mn}. For Theorem \ref{thm:run_time}, we consider the whole process of converting the vector into the QTT-format, computing the
convolution, then converting the convolution in the QTT format back into a vector, as is demonstrated in \aref{alg:qtt_conv}. For the last
step, to convert a tensor in the TT-format back into the standard format, we use the 'full' algorithm from the Matlab toolbox
\textbf{oseledets/TT-Toolbox}. This is given in \aref{alg:full}. We then reshape this tensor into a vector with a bit of run time.

\begin{algorithm}[ht!]
    \footnotesize
    \SetKwInOut{Input}{input}
    \SetKwInOut{Output}{output\,}
    \SetAlgoLined
    \Input{$\tA^{(1)},\tA^{(2)},...,\tA^{(K)}$, and size of output tensor [$M_1,\ldots,M_k$]}
    \Output{$\tA\in\C^{M_1\times\ldots\times M_k}$}
    Let $\matA = \tA^{(1)}$\\
     \For{k=2,...,K}{
      $\matA = \text{reshape}(\matA,[\frac{(|\matA|)}{r_{k-1}},r_{k-1}])$\\
      $\matB = \text{reshape}(\tA^{(k)},[r_{k-1},2r_{k}])$\\
      $\matA = \matA\matB$\\
     }
     $\tA = \text{reshape}(\matA,[M_1,\ldots,M_k])$
     \caption{Full}
     \label{alg:full}
\end{algorithm}

\begin{lemma}\label{lemma:mn}
Let $\tA\in\R^{2\times\ldots\times 2}$ be a $K$-mode tensor. Let $\{\matA^{\{k\}}\}_{k=1}^{K-1}$ be the unfolding matrices of $\tA$ in the max
rank TT-SVD algorithm with a max rank of $R_{\max}$ and with each $\matA^{\{k\}}\in\C^{m_k\times n_k}$. Then
\begin{equation*}
m_k = 2r_{k-1}\leq 2R_{\max} \quad \text{and} \quad n_k=2^{K-k}.
\end{equation*}
\begin{proof}
Since $M_k=2$ for all $k$, the proof for $m_k = 2r_{k-1}\leq 2R_{\max}$ is trivial by the first line inside the for loop in Algorithm
\ref{alg:tt_svd}. For $n_k$, we do a proof by induction. First, note that
$|\matA^{\{1\}}|=2^K$ and $r_0=1$, thus
\begin{equation*}
n_1 = \frac{|\matA^{\{1\}}|}{2r_0}=\frac{2^K}{2} = 2^{K-1}.
\end{equation*}
Assume $n_\ell = 2^{K-\ell}$ for all $1\leq \ell\leq k-1$. Then,
$$
n_k= \frac{|\matA^{\{k\}}|}{2r_{k-1}}=\frac{|\matSi_{k-1}\matV_{k-1}^*|}{2r_{k-1}}=\frac{r_{k-1}n_{k-1}}{2r_{k-1}}=\frac{n_{k-1}}{2}
=\frac{2^{K-(k-1)}}{2}=2^{K-k}.
$$
Thus, we get
\begin{equation*}
m_k = 2r_{k-1}\leq 2R_{\max} \quad \text{and} \quad n_k=2^{K-k}.
\end{equation*}
\end{proof}
\end{lemma}

\begin{thm}\label{thm:run_time}
Let $\bbf_\xi,\bbg\in\R^{2^{K-1}-1}$ for some positive integer $K$. Then the computational complexity, $C_{\text{QTT-conv}}$, of approximating
the convolution $\bbf_\xi*\bbg$ with the max rank TT-SVD and max rank QTT-SVD algorithms described above is
\begin{equation*}
C_{\text{QTT-conv}}\leq \mathcal{O}(R^2_{\max} 2^K),
\end{equation*}
where $R_{\max}$ is the prescribed max rank for both the TT-SVD algorithms and the QTT-FFT algorithm.
\end{thm}
\begin{proof}
We show that the computational complexity is dominated asymptotically by the max rank TT-SVD algorithms and the full tensor algorithm. First,
let $C_{svd}$ be the computational cost of the SVD in big $\mathcal{O}$ notation. Then, for a matrix $\matA\in\C^{m\times n}$,
$C_{svd}(\matA)=\mathcal{O}(mn\min(m,n)$). Note that, in Algorithm \ref{alg:tt_svd} (as well as in our max rank modifications), the
computational complexity is dominated by the SVD algorithm. Denote the unfolding matrices at the $k$th iterations as
$\matA^{\{k\}}\in\C^{m_k\times n_k}$. Hence, the computational cost of the max rank TT-SVD algorithm is
\begin{equation*}
\begin{split}
    \sum_{k=1}^{K-1} C_{\text{svd}}(\matA^{\{k\}}) &= \sum_{k=1}^{K-1} \mathcal{O}(m_kn_k\min(m_k,n_k))\\
    &\leq \sum_{k=1}^{K-1} \mathcal{O}((2R_{\max})^2 2^{K-k})\\
    & = 4R_{\max}^2 \sum_{k=1}^{K-1} \mathcal{O}(2^{k})\\
    & = 4R_{\max}^2 \mathcal{O}(2^{K}-2)\\
    & = \mathcal{O}(R^2_{\max}2^K).
\end{split}
\end{equation*}
From \cite{DKS12}, we have that for the QTT-FFT and QTT-iFFT algorithms, the computational complexity is $\mathcal{O}(K^2R^3_{\max})$.
In Algorithm \ref{alg:full}, the computational complexity comes from the multiplication $\matA\matB$ in every loop. For the $k$th loop,
$\matA\in\C^{2^{k-1}\times r_{k-1}}$ and $\matB\in\R^{r_{k-1}\times 2r_{k}}$ for $k=2,\ldots,K$, thus the computational complexity is
proportional to the cost of multiplying $\matA$ by $\matB$, i.e.,
\begin{equation*}
\begin{split}
    C_{\text{full}} &= \sum_{k=2}^K \mathcal{O}(2^{k-1}r_{k-1}2r_{k})\\
    & \leq R^2_{\max} \sum_{k=2}^{K} \mathcal{O}(2^{k})\\
    & = R^2_{\max}\mathcal{O}(2^{K+1}-4) \\
    & = \mathcal{O}(R^2_{\max}2^K).
\end{split}
\end{equation*}
Hence, the total computational complexity is
\begin{equation*}
    \mathcal{O}(R^2_{\max}2^K)+\mathcal{O}( K^2R^3_{\max})+\mathcal{O}(R^2_{\max}2^K)=\mathcal{O}(R^2_{\max}2^K).
\end{equation*}
\end{proof}
\noindent
For the randomized SVD, we have the computational complexity
$C_{rsvd}(A^{\{k\}})=\mathcal{O}(m_kn_k(R_{\max}+p))=\mathcal{O}(2^{K-k}R_{\max}(R_{\max}+p))$. Thus, the run time for the convolution with a max
rank TT-RSVD is similar when $p$ is small. In $D$ spatial dimensions, we can obtain a similar result but by replacing $K$ with $DK$ in the max
rank TT-SVD algorithm and the full tensor algorithm, and the QTT-FFT algorithm is $\mathcal{O}(DK^2R^3_{\max})$. Hence, the total run time
complexity in $D$ spatial dimensions is $\mathcal{O}(R^2_{\max}2^{DK})$.

\section{Denoising}
\label{sec:denoising}
It is well known that the SVD can remove noise from matrix data, as seen in \cite{EK19,JY11}, but little research has been done in denoising
with tensor decompositions. In \cite{LZT19} and \cite{NYAHIS20}, the Tucker decomposition was used to help remove noise from point cloud data
and electron holograms, respectively. In \cite{GCCA20}, it was shown that the TT-decomposition might have some advantages to denoising as
opposed to the Tucker decomposition. This is because a low-rank Tucker matrix guarantees a low TT-rank for the data. However, the converse
statement is not always true.

Let $\tF$ be the low TT-rank tensor representing $\bbf$ in the QTT format. Then for some core tensors
$\tF^{(k)}\in\R^{r_{k-1}\times 2\times r_k}$ with tensor slices $\tF^{(k)}(:,i_k,:)=\bbf^{(k)}_{i_k}\in\R^{r_{k-1}\times r_k}$, $i_k=0,1$. Each
element of $\tF$ can be represented in the TT format as
\begin{equation*}
    \tF(i_1,\ldots,i_k)=\bbf^{(1)}_{i_1}\ldots \bbf^{(K)}_{i_K}, \quad i_k=0,1, \quad k=1,\ldots,K,
\end{equation*}
where each $\bbf^{(k)}_{i_k}$ is a low rank matrix. In practice, it is unlikely the data collected has a low-rank TT decomposition since almost
all real radar data has noise due to hardware limitations or other signals interfering with the data. Instead, we have the noisy data
$\bbf_\xi$ whose tensor representation is
\begin{equation*}
    \tF_\xi = \tF + \bxi,
\end{equation*}
where $\bxi$ is the realization of the random noise in the TT format. The tensor $\tF_\xi$ almost surely has full TT-rank when represented
exactly in the QTT format. Ideally, we would like to be able to find an approximate TT decomposition $\tilde{\tF}$ with TT-cores
$\tilde{\tF}^{(k)}$, $k=1,\ldots,K$, using the noisy data such that $\tilde{\tF}^{(k)} \approx \tF^{(k)}$. However, it is hard to guarantee any bound on this. We argue, though, that by using our proposed methods when given the noisy data $\tF_\xi$, we can find a TT decomposition
$\tilde{\tF}$ with low rank such that $\tilde{\tF}\approx\tF$.

Consider the first iteration of the for loop of algorithm \ref{alg:tt_svd}, with $\tA=\tA_0+\tA_\xi$ as the sum of a smooth tensor ($\tA_0$)
and a noisy tensor ($\tA_\xi$). Then, after it is reshaped, we obtain the matrix
\begin{equation*}
    \matA^{\{1\}} = \matA^{\{1\}}_0+\matA^{\{1\}}_\xi,
\end{equation*}
where $\matA^{\{1\}}_0$ is a low rank matrix and $\matA^{\{1\}}_\xi$ is added noise. Let $\matA^{\{1\}} = \matU\matSi \matV^*+\matE$ be the
truncated SVD of $\matA^{\{1\}}$ and $\matA^{\{1\}}_0 = \matU_0\matSi_0 \matV_0^*$ be the SVD of $\matA^{\{1\}}_0$. Note that
$\matU\matSi \matV^*\approx \matA_0^{\{1\}}$ does not imply that $\matU\approx \matU_0$, and thus the TT-core $\tA^{(1)}$ is not guaranteed to
be approximately equal to $\tA^{(1)}_0$, where $\tA^{(1)}_0$ is the first TT-core of $\tA_0$. However, if we let $\tA^2=\tA$ on the second
iteration of the loop in \aref{alg:tt_svd} (and similarly for $\tA_0$), we do get that the elements of the tensor contraction
$\tA^{(1)}\circ \tA^2 \approx \tA^{(1)}_0\circ\tA^2_0$. Similarly, if we can approximate the noise-free component on every iteration of the for
loop, we obtain an approximation for the tensor $\tA_0$. While we do not have a theoretical bound on this error, our experiments in
\sref{sec:ex_conv} show that this method works well at removing the noise. Since our method computes multiple SVDs, it can reduce a lot more
noise than if we just did a single SVD and can do so without excessive smearing.

\section{Numerical simulations}\label{sec:ex_conv}
This section presents some examples in one and two spatial dimensions. The original code for the  TT-decompositions and the QTT-FFT algorithms
comes from the Matlab toolbox  \textbf{oseledets/TT-Toolbox}. We have modified it accordingly for the \textbf{max rank TT-SVD},
\textbf{max rank TT-RSVD}, and \textbf{SV drop off TT-SVD} algorithm, as discussed in \sref{sec:computing}. For all our examples, we compare
the run time and errors of computing the convolution \eqref{conv} using several methods. The error for every example is the $l_2$ relative
error
\begin{equation}\label{err_conv}
    E_{2}(\bbI)=\frac{\|\bbI-\bbI_\text{ref}\|_2}{\|\bbI_\text{ref}\|_2},
\end{equation}
where in $D$ spatial dimensions
\begin{equation*}
    \|\bbI\|_2 = \sqrt{\frac{1}{N^D}\sum_{\mj=\0}^{\mN-\1} |\bbI_{\mj}|^2}.
\end{equation*}
The reference solution, $\bbI_\text{ref}$, is the discrete convolution \eqref{dconv} computed without any noise. In all of the examples,
we compare our methods against computing the convolution with the randomized TT-SVD algorithm from \cite{HSW18}, as well as computing the true
noisy convolution with FFT. In two space dimensions, we also approximate the convolution using a low matrix rank approximation to the noisy
data $\bbf_\xi$, where the truncated rank is determined by the actual matrix rank of $\bbf$.

For all of these examples, we use the normalized sinc imaging kernel that corresponds to the GAF (\ref{eq:GAF}) truncated to a sufficiently
large interval $[-L,L]$:
\begin{equation} \label{sinc_kern1}
g(x) = \frac{\sinc(\pi\frac{x}{\Delta_x})}{\int_{-L}^L \sinc(\pi\frac{x}{\Delta_x})\; dx},\quad x\in[-L,L]
\end{equation}
for $D=1$, and
\begin{equation} \label{sinc_kern2}
g(x,y) = \frac{\sinc(\pi\frac{x}{\Delta_x})\sinc(\pi\frac{y}{\Delta_y})}{\iint_{-L}^L
\sinc(\pi\frac{x}{\Delta_x})\sinc(\pi\frac{y}{\Delta_y})\; dxdy},\quad (x,y)\in[-L,L]\times[-L,L]
\end{equation}
for $D=2$, where the resolution $\Delta_x$ in  \eqref{sinc_kern1} and $\Delta_x=\Delta_y$ in \eqref{sinc_kern2} is a given parameter. The one-
dimensional kernel \eqref{sinc_kern1} for $\Delta_x= 0.04\pi$ is shown in Figure~\ref{fig:kernel1}.

In Table \ref{tab:qtt_rel_err}, we present the relative error for each example for $K=20$ when $D=1$, and $K=10$ when $D=2$. In this table,
the convolution $\bbf_\xi*\bbg$ is denoted by $\bbI_\xi$ and  computed using the FFT algorithm,
the QTT-convolution computed with the \textbf{max rank TT-SVD} algorithm is denoted by $\bbI_{QTT_0}$, the QTT-convolution computed with the
\textbf{max rank TT-RSVD} is denoted by $\bbI_{QTT_r}$, and the convolution computed using the \textbf{SV drop off TT-SVD} algorithm is denoted
by $I_\delta$. In turn, the convolutions computed using the randomized TT-decomposition are denoted by $\bbI_{RTT}$, and in two dimensions, the
convolution computed using low-rank approximations of $\bbf$ is denoted by $\bbI_{lr}$. For $I_\delta$ and $\bbI_{lr}$, we also denote what
parameter $\delta$ and  truncation matrix rank $R$ are used, respectively, for each example using a subscript of the error.
\begin{figure}[ht!]
\centering
\includegraphics[trim=0 0 0 0,clip,width=.5\linewidth]{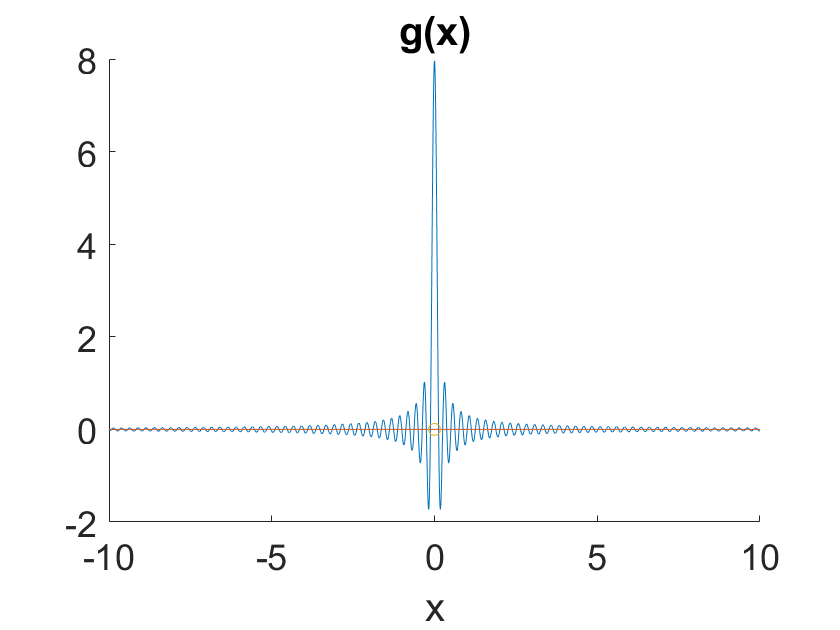}
    \caption{Kernel function \eqref{sinc_kern1} with $\Delta_x= 0.04\pi$.}
    \label{fig:kernel1}
\end{figure}

For each example, we show the TT-ranks of the original function without noise, $\bbf$, in the QTT format given by the tensor $\tF$. This QTT
approximation is computed with Algorithm \ref{alg:tt_svd} with the tolerance $\varepsilon=10^{-10}$. We compute these TT-ranks for $K=20$ when
$D=1$ and $K=10$ when $D=2$. However, it is worth noting that these TT-ranks do not change much for any data size.
Notice that the max TT-ranks we choose for our algorithms are less than the TT-ranks of $\bbf$ from \aref{alg:tt_svd}, yet still provide a
reasonable estimate.
\begin{table}[ht!]
    \centering
    \caption{$l_2$-norm relative error for $K=20$ for examples 1 and 2, and $K=10$ for example 3.}
    \begin{tabular}{|c|c|c|c|c|c|c|}
    \hline
    example   &  $\bbI_\xi$   & $\bbI_{QTT_0}$ & $\bbI_{QTT_r}$ & $\bbI_{\delta}$ & $\bbI_{RTT}$ &  $\bbI_{lr}$\\
    \hline
    \hline
     $1$ & $0.0383$ & $\textbf{0.0028}$  & $0.0102$ & $0.0280_{\delta=0.02}$ & $0.0430$ & -\\ \hline
     $2$ & $0.0131$  & $\textbf{0.0011}$ & $0.0075$ & $0.0068_{\delta=0.01}$ & $0.0201$ & -\\ \hline
     $3$ & $0.1142$  & $\textbf{0.0151}$ & $0.0447$ & $0.1650_{\delta=0.09}$ & $0.1534$ & $0.0470_{rank=23}$\\ \hline
    \end{tabular}
    \label{tab:qtt_rel_err}
\end{table}

\subsection{Example 1}\label{ssec:qtt_e1}
For this example, let
\begin{equation*}
    f(x) = e^{-(\frac{3x}{10})^2}(0.4\sin(8\pi x)-0.7\cos(6\pi x)),\quad x\in[-10,10],
\end{equation*}
and
\begin{gather*}
      (f_{\xi})_j = f(x_j)+\xi_j,\\  x_j=-10+\frac{\Delta x}{2}+j\Delta x,\quad j=0,\ldots,N-1, \quad \Delta x = \frac{20}{N}, \quad N=2^{K-1}-1,
\end{gather*}
with $\xi_j\sim\mathcal{N}(0,0.02)$. We also set the resolution $\Delta_x=4\Delta x$  in $\eqref{sinc_kern1}$, where $\Delta x$ is the size of the spatial discretization. Thus, the width of the main lobe of the sinc is $8\Delta x$ on the x-axis.

As we can see in Figure \ref{fig:conv_e1}, the FFT-QTT algorithm removed much of the noise in the data compared to the true convolution.
For $K=20$, we also tried computing the convolution using the original TT-SVD algorithm given in \aref{alg:tt_svd} with multiple values of $\varepsilon$. The smallest error, as defined in \eqref{err_conv}, occurred when $\varepsilon = 0.01$ and gave the relative error of $E_{2}(\bbI) = 0.03202$. This is close to the error of the true convolution of the noisy data and took over $100$ seconds to compute. However, as we can see in Table \ref{tab:qtt_rt_e1}, the run times for all of our methods on the same grid took less than a second. This indicates that the original TT-SVD algorithm is practically unsuitable for removing data noise.

The max TT-rank of the discretization of $f(x)$ in the QTT format, $\tF$, is $17$, yet we were able to achieve our approximation using a max rank of $R_{max}=10$ for the \textbf{max rank TT-SVD} and \textbf{max rank TT-RSVD} algorithms and $\hat{R}_{max}=15$ for the QTT-FFT algorithm. Thus, even if we do not know the exact TT-rank, we can still compute a good approximation.

Table \ref{tab:qtt_rt_e1} shows run times for different grid sizes for each method. We can see that computing the convolution with FFT is faster than our methods for these values of $K$. However, the convolution with our QTT methods gets closer to the FFT run time as $K$ increases. This is shown in the last column of Table \ref{tab:qtt_rt_e1} where we see the ratio of the \textbf{max rank TT-SVD} convolution method to the FFT convolution method is getting smaller as $K$ grows. This helps verify our theoretical result that for some constant max rank $R_{max}$ (and $\hat{R}_{max}$), the \textbf{max rank TT-SVD} convolution method is asymptotically faster than computing the convolution with FFT. The amount of data needed for our method to outperform the FFT method may be impractical for most real-world applications in 1-2 spatial dimensions.

\begin{figure}[ht!]
    \centering
    \begin{minipage}{.45\textwidth}
    \centering
    \includegraphics[trim=45 0 0 0,clip, height=0.85\textwidth,width=1.1\textwidth]{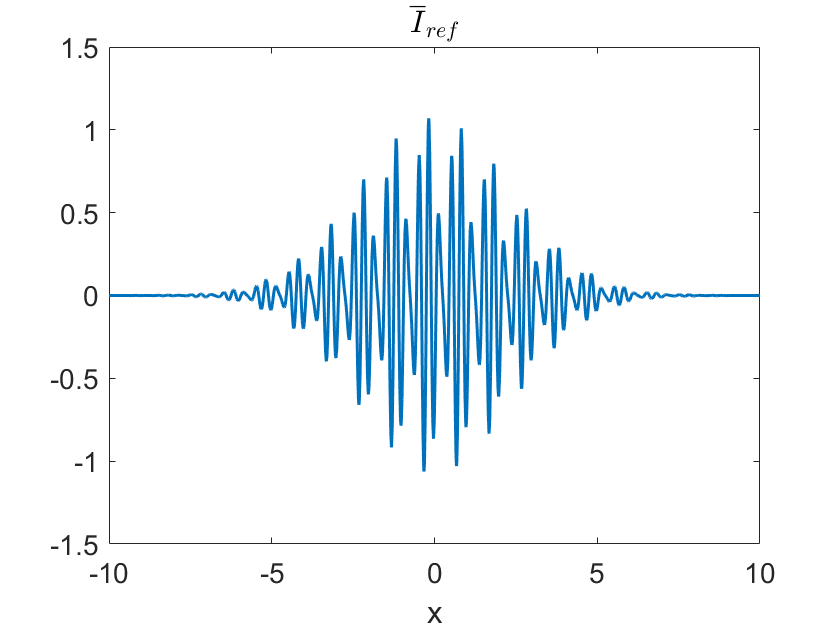}
    \end{minipage}
    \begin{minipage}{.45\textwidth}
    \centering
    \includegraphics[trim=0 0 45 0,clip, height=0.85\textwidth,width=1.1\textwidth]{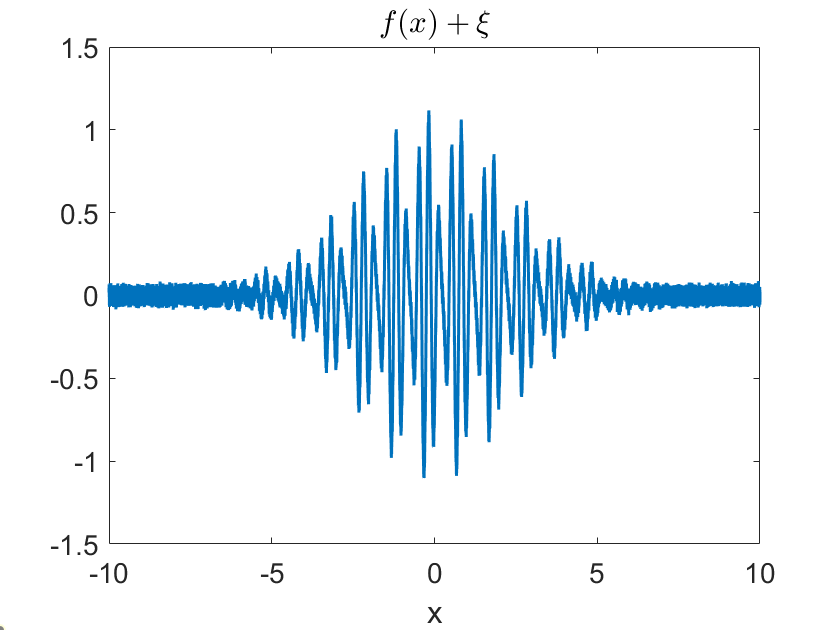}
    \end{minipage}
    \begin{minipage}{.45\textwidth}
    \centering
    \includegraphics[trim=45 0 0 0,clip, height=0.85\textwidth,width=1.1\textwidth]{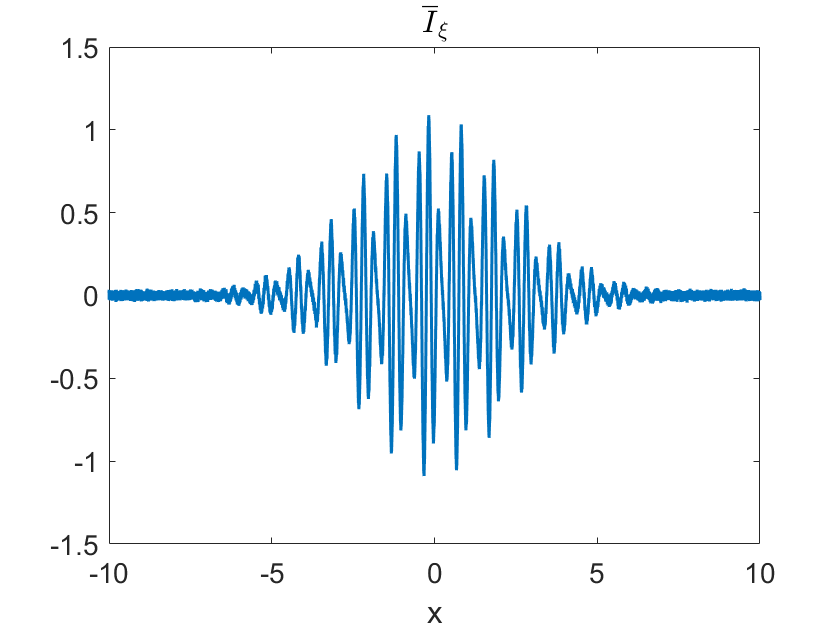}
    \end{minipage}
    \begin{minipage}{.45\textwidth}
    \centering
    \includegraphics[trim=0 0 45 0,clip, height=0.85\textwidth,width=1.1\textwidth]{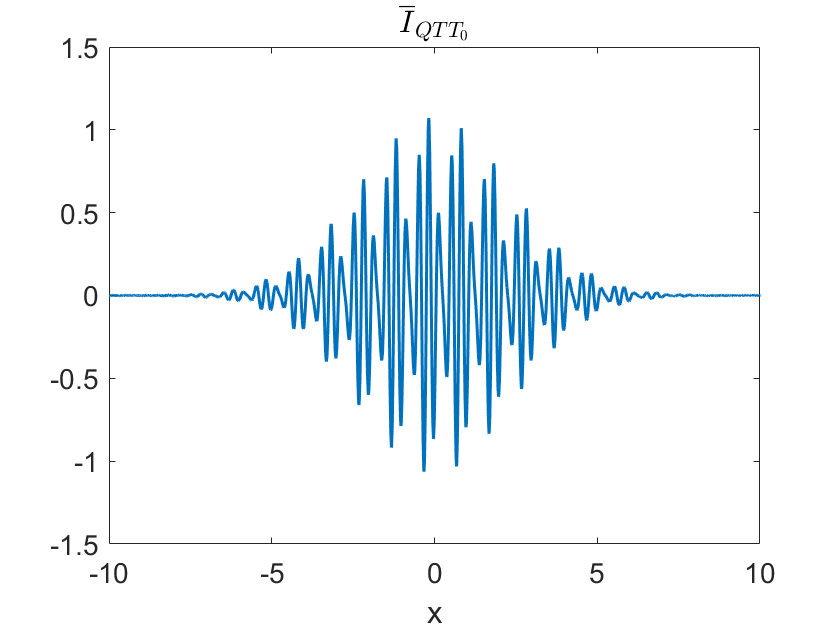}
    \end{minipage}
    \begin{minipage}{.45\textwidth}
    \centering
    \includegraphics[trim=45 0 0 0,clip, height=0.85\textwidth,width=1.1\textwidth]{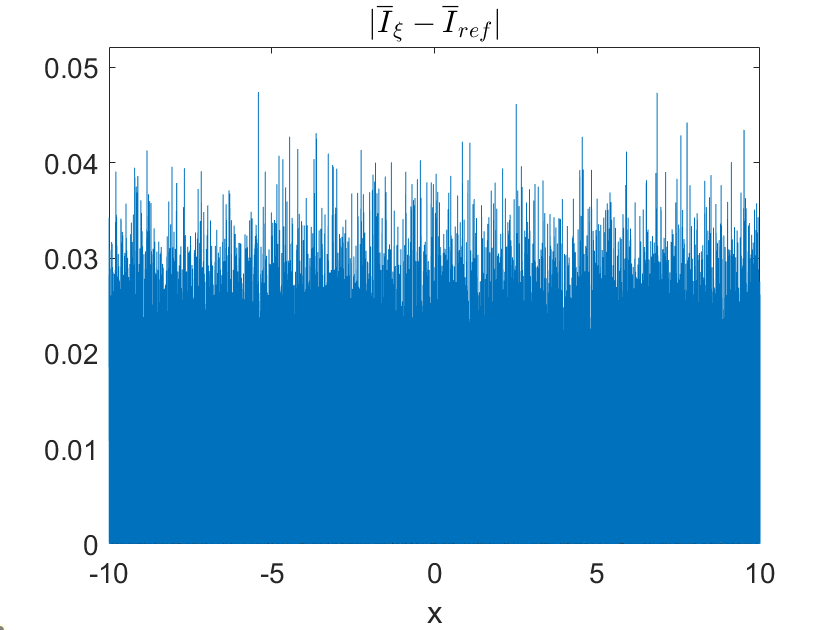}
    \end{minipage}
    \begin{minipage}{.45\textwidth}
    \centering
    \includegraphics[trim=0 0 45 0,clip, height=0.85\textwidth,width=1.1\textwidth]{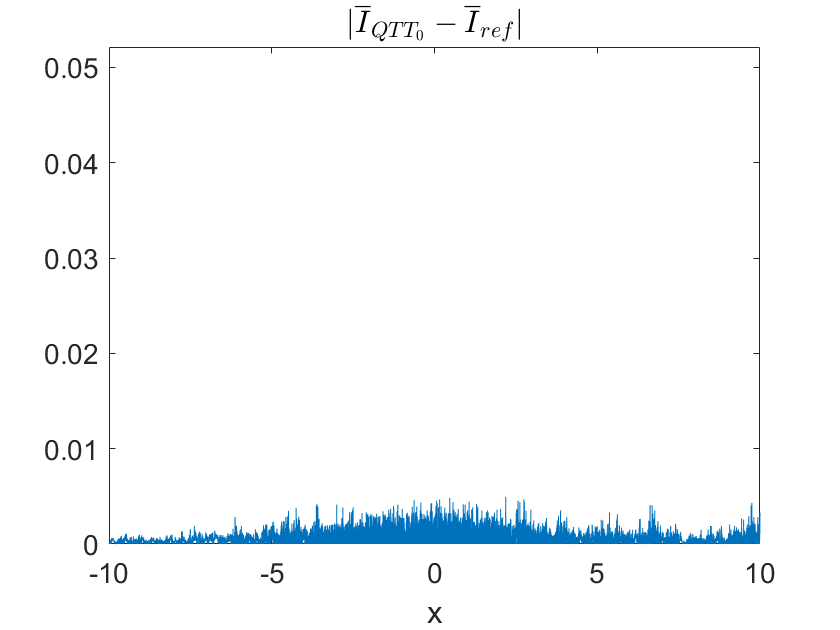}
    \end{minipage}
    \caption{\textit{Top Left:} True convolution of data without noise, $\bbI$. \textit{Top Right:} Function data with noise, $\bbf_\xi$.
    \textit{Middle Left:} True convolution of data with noise, $\bbI_\xi$. \textit{Middle Right:} Convolution using the \textbf{max rank TT-SVD} algorithm, $\bbI_{QTT_0}$.
    \textit{Bottom Left:} Absolute error of $\bbI_\xi$. \textit{Bottom Right:} Absolute error of $\bbI_{QTT_0}$.}
    \label{fig:conv_e1}
\end{figure}

\begin{table}[ht!]
    \centering
    \caption{Run time (seconds): Example 1 convolutions.}
    \begin{tabular}{|c|c|c|c|c|c|c|}
    \hline
    K  &  $\bbI_\xi$   & $\bbI_{QTT_0}$ & $\bbI_{QTT_r}$ & $\bbI_{\delta}$ & $\bbI_{RTT}$ & $\bbI_{QTT_0}/\bbI_\xi$\\
    \hline
    \hline
     $16$ & $\textbf{0.005}$ & $0.325$  & $0.369$ & $1.485_{\delta=0.02}$ & $0.414$ & 65\\ \hline
     $20$ & $\textbf{0.067}$  & $0.653$ & $0.694$ & $0.479_{\delta=0.02}$ & $0.609$ & 9.7463\\ \hline
     $24$ & $\textbf{1.21}$  & $4.41$ & $4.86$ & $3.10_{\delta=0.02}$ & $2.80$ & 3.6446\\ \hline
     $26$ & $\textbf{5.77}$  & $17.75$ & $19.68$ & $13.93_{\delta=0.02}$ & $10.97$ & 3.0763\\ \hline
     $28$ & $66.6$  & $95.5$ & $108.7$ & $79.0_{\delta=0.02}$ & $\textbf{62.49}$ & 1.4339\\ \hline
    \end{tabular}
    \label{tab:qtt_rt_e1}
\end{table}

\begin{table}[ht!]
    \centering
    \caption{Data storage for Example 1.}
    \begin{tabular}{|c|c|c|}
    \hline
    K   &  $f_\xi$   & $\tF_{\xi}$\\
    \hline
    \hline
     $16$ & $65{,}536$ & $2088$ \\ \hline
     $20$ & $1{,}048{,}576$ & $2888$ \\ \hline
     $24$ & $16{,}777{,}216$ & $3688$ \\ \hline
     $26$ & $67{,}108{,}864$ & $4088$ \\ \hline
     $28$ & $268{,}435{,}456$ & $4488$ \\ \hline
    \end{tabular}
    \label{tab:data_storage_e1}
\end{table}

Table \ref{tab:data_storage_e1} shows the number of elements to represent the data $\bbf_\xi$ fully versus how many elements are required to store the data in the QTT-format with a prescribed max rank of $R_{max}=10$, $\tF_{\xi}$, in Example 1. As we can see, storing all the elements takes a lot of data and grows exponentially in $K$, while storing the elements in the QTT format takes a lot less data and only grows linearly in $K$. These values for the QTT-data storage can be found by looking at the size of the core tensors. For the tensor $\tF_{\xi}$ in the QTT format and with a max TT-rank of $R_{max}=10$, we have the TT-cores
\begin{align*}
    & \tF^{(1)}_{\xi},\tF^{(K)}_{\xi}\in\R^{1\times2\times2},\\
    & \tF^{(2)}_{\xi},\tF^{(K-1)}_{\xi}\in\R^{2\times2\times4},\\
    & \tF^{(3)}_{\xi},\tF^{(K-2)}_{\xi}\in\R^{4\times2\times8},\\
    & \tF^{(4)}_{\xi},\tF^{(K-3)}_{\xi}\in\R^{8\times2\times10},\\
    & \tF^{(k)}_{\xi}\in\R^{10\times2\times10},\quad k=5,\ldots,K-4.
\end{align*}
Thus, the number of elements, $N_{el}$, that make up this QTT tensors is:
\begin{equation*}
    N_{el} = 2(1\times2\times2)+2(2\times2\times4)+2(4\times2\times8)+2(8\times2\times10)+(K-8)(10\times2\times10).
\end{equation*}

The \textbf{max rank TT-RSVD} algorithm is not able to produce results as good as the \textbf{max rank TT-SVD} (see Table \ref{tab:qtt_rel_err} for relative error comparison and Table \ref{tab:qtt_rt_e1} for a run time comparison) but is still able to produce a reasonably low error. While the run time for the  \textbf{max rank TT-SVD} is faster than the \textbf{max rank TT-RSVD} for all of our methods, the \textbf{max rank TT-RSVD} can be faster for tensors with larger mode sizes. This is due to the SVD in \textbf{max rank TT-SVD} algorithm with mode sizes, $M_k$, may be computed on a matrix with $m_k=M_kR_{\max}$ rows. In contrast, for the \textbf{max rank TT-RSVD} algorithm, the SVD is computed on a matrix with $m_k=R_{\max}+p$ rows when $M_k=2$ (such as for the QTT decomposition). The difference in the sizes of $m_k$ does not make up for the extra amount of work the RSVD algorithm does. Although this paper focuses on the QTT-decomposition and thus $M_k=2$, we believe this is important to note as the \textbf{max rank TT-RSVD} algorithm can speed up the TT-decomposition for higher mode tensor data and still produce accurate approximations. We verify this by computing the \textbf{max rank TT-SVD} algorithm and the \textbf{max rank TT-RSVD} algorithm on a tensor with $K=8$ modes with each mode of size $M_k=10$, $k=1,\ldots,K$. Each element of this tensor is taken from the uniform distribution $\mathcal{U}[0,1)$. The max rank TT-SVD algorithm took $9.57$ seconds, and the \textbf{max rank TT-RSVD} algorithm only took $5.12$ seconds, almost half the time of the \textbf{max rank TT-SVD} algorithm.

\subsection{Example 2}\label{ssec:qtt_e2}

If we were to choose a coarser resolution for the example of Section~\ref{ssec:qtt_e1}
(i.e.,  a wider sinc function), we could reduce the noise using the standard convolution at the cost of smoothing out the solution's peaks. Doing this gives similar results for the true convolution and with our methods (\sref{sec:computing}). In this section, we show an example where the ground reflectivity is very oscillatory. Here, the resolution $\Delta_x$ determined by the GAF must be small (i.e., the sinc function must be ``skinny'').  Otherwise, if the sinc window is close to or larger than the characteristic scale of variation of the ground reflectivity,  then the convolution can smooth out the actual oscillations instead of just the noise, losing most of the information in $f$.

We choose the ground reflectivity as
\begin{equation*}
    f(x) = e^{-(3x)^2}(0.9\sin(\frac{2x\pi}{5\Delta x})+1.4\cos(\frac{x\pi}{3\Delta x})),\quad x\in[-1,1],
\end{equation*}
and
\begin{gather*}
      (f_{\xi})_j = f(x_j)+\xi_j,\\  x_j=-1+\frac{\Delta x}{2}+j\Delta x,\quad j=0,\ldots,N-1, \quad \Delta x = \frac{2}{N}, \quad N=2^{K-1}-1,
\end{gather*}
with $\xi_j\sim\mathcal{N}(0,.01)$. We use $\Delta_x=2\Delta x$ and the max TT-rank of the discretization of the smooth function $f(x)$ in the QTT-format is $26$.

\begin{figure}[ht!]
    \centering
    \begin{minipage}{.45\textwidth}
    \centering
    \includegraphics[trim=10 0 0 0,clip, height=0.95\textwidth,width=1.1\textwidth]{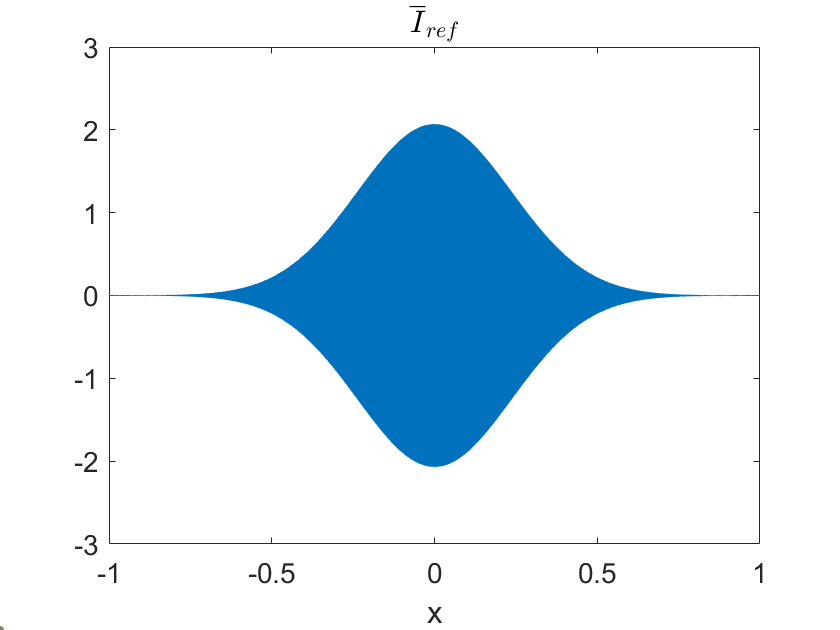}
    \end{minipage}
    \begin{minipage}{.45\textwidth}
    \centering
    \includegraphics[trim=0 0 10 0,clip, height=0.95\textwidth,width=1.1\textwidth]{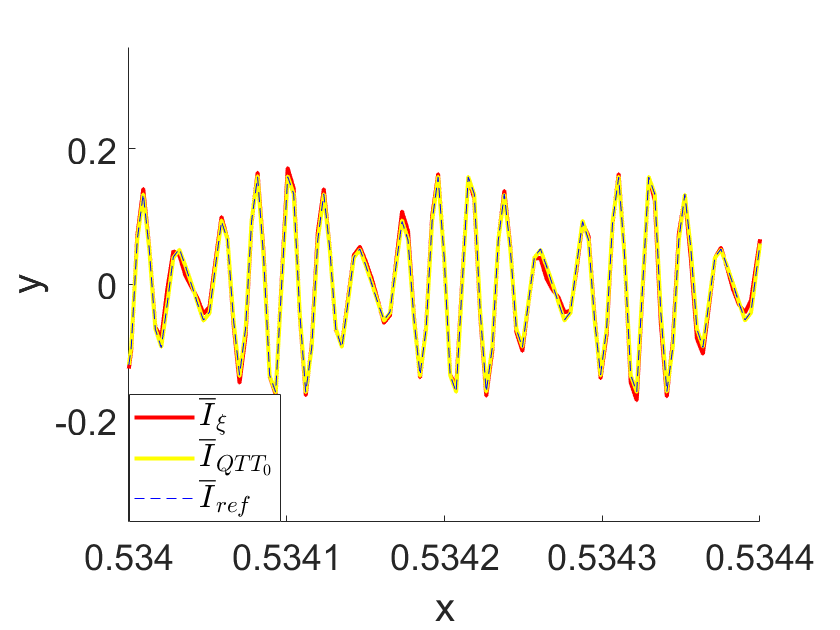}
    \end{minipage}
    \begin{minipage}{.45\textwidth}
    \centering
    \includegraphics[trim=10 0 0 0,clip, height=0.95\textwidth,width=1.1\textwidth]{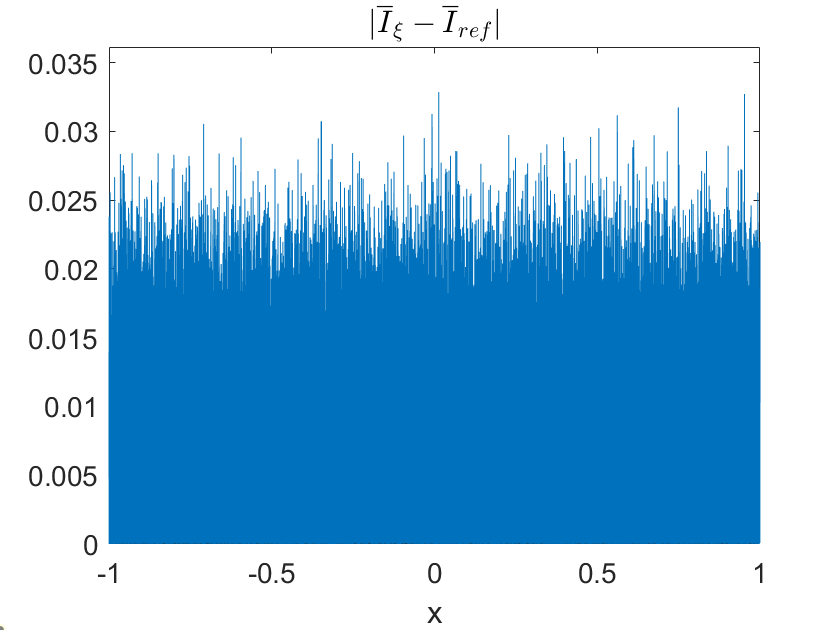}
    \end{minipage}
    \begin{minipage}{.45\textwidth}
    \centering
    \includegraphics[trim=0 0 10 0,clip, height=0.95\textwidth,width=1.1\textwidth]{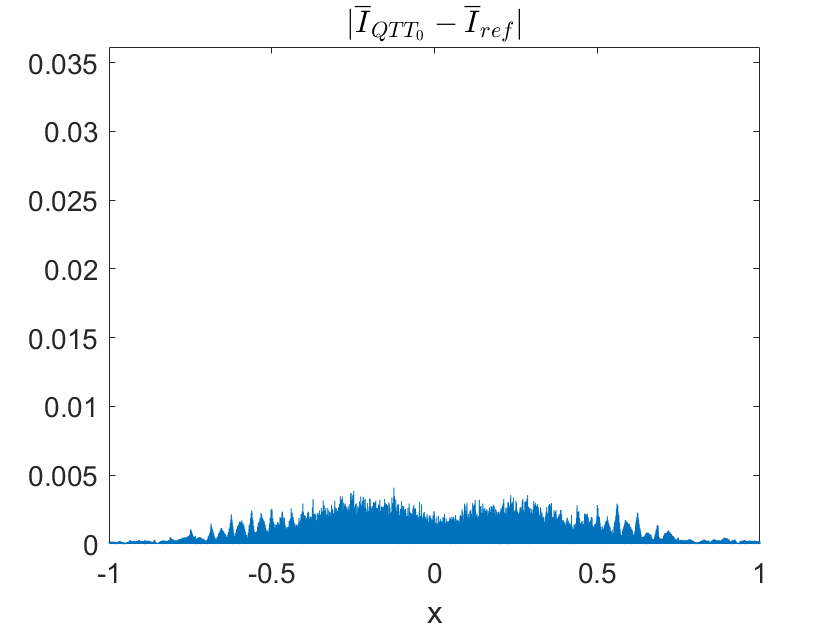}
    \end{minipage}
    \caption{\textit{Top Left:} True convolution of data without noise, $\bbI$. \textit{Top Right:} Zoomed in graph of $\bbI_{\xi}$, $\bbI_{QTT_0}$, and $\bbI_{ref}$.
    \textit{Bottom Left:} Absolute error of $\bbI_\xi$. \textit{Bottom Right:} Absolute error of $\bbI_{QTT_0}$.}
    \label{fig:conv_e2}
\end{figure}

Again, we use the max ranks of $R_{max} = 10$ and $\hat{R}_{max}=15$ for the \textbf{max rank TT-SVD} and QTT-SVD algorithms, respectively.
As the function is too oscillatory to see a lot of helpful information in the full graph (see top left of Figure \ref{fig:conv_e2}), we show a zoomed-in plot of the graph of $\bbI_\xi$, $\bbI_{QTT_0}$, and $\bbI_\text{ref}$ (see top right of Figure \ref{fig:conv_e2}). While there is some error, the QTT-FFT convolution agrees with the true convolution $\bbI_{ref}$ very well, whereas $\bbI_\xi$ has a more considerable noticeable difference. This is verified by the graphs of the absolute error given in Figure \ref{fig:conv_e2}, where the bottom left shows the error for $\bbI_{\xi}$, and the bottom right shows the error for $\bbI_{QTT_0}$.

\subsection{Example 3}\label{ssec:qtt_e3}
Let
\begin{gather*}
    f(x,y) = e^{-((2x)^2+(2y)^2)}(\sin(2\pi x)-\cos(7\pi y)+\cos(4\pi xy)-\sin(3\pi x y),\\
    (x,y)\in[-1,1]\times[-1,1]
\end{gather*}
and
\begin{gather*}
      (\mf_{\xi})_{j,k} = f(x_j,y_k)+\xi_{j,k},\\  x_j=-1+\frac{\Delta x}{2}+j\Delta x,\quad y_k=-1+\frac{\Delta y}{2}+k\Delta y, \\ j,k=0,\ldots,N-1, \quad \Delta x = \Delta y = \frac{2}{N}, \quad N=2^{K-1}-1,
\end{gather*}
with $\xi_{j,k}\sim\mathcal{N}(0,0.1)$. We have $\Delta_x=\Delta_y=2\Delta x$. Thus, the diameter of the main lobe of the sinc is $4\Delta x$ on the xy-plane. Here, we show a 2D example whose discretization of a smooth function $f$ has a matrix rank of $23$ and a TT-rank of 26 when represented in the QTT format. We still use the ranks $R_{max}=10$ and $\hat{R}_{max} = 15$ for our \textbf{max rank TT-SVD} (\textbf{max rank TT-RSVD}) and max rank QTT-FFT. Thus, our TT-ranks are much smaller than the true TT-ranks. In Figure \ref{fig:conv_e3} and \tref{tab:qtt_rel_err}, notice that our method can still capture the shape of the original function with an error that is an order of magnitude smaller than the error from the true convolution using FFT. The plots on the bottom of Figure \ref{fig:conv_e3} are a side view of the error graphs, as it is easier to compare the errors in this view.
The 2D examples are similar to the previous test case. Thus, it is reasonable to assume our method works about the same regardless of the spatial dimension.

\begin{figure}[ht!]
    \centering
    \begin{minipage}{.45\textwidth}
    \centering
    \includegraphics[trim=40 0 0 0,clip, height=0.85\textwidth,width=1.1\textwidth]{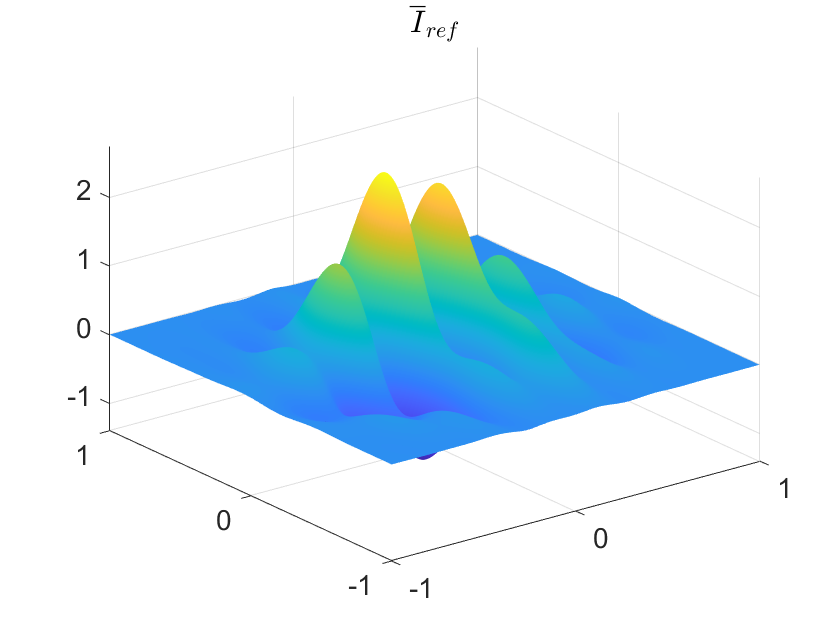}
    \end{minipage}
    \begin{minipage}{.45\textwidth}
    \centering
    \includegraphics[trim=0 0 40 0,clip, height=0.85\textwidth,width=1.1\textwidth]{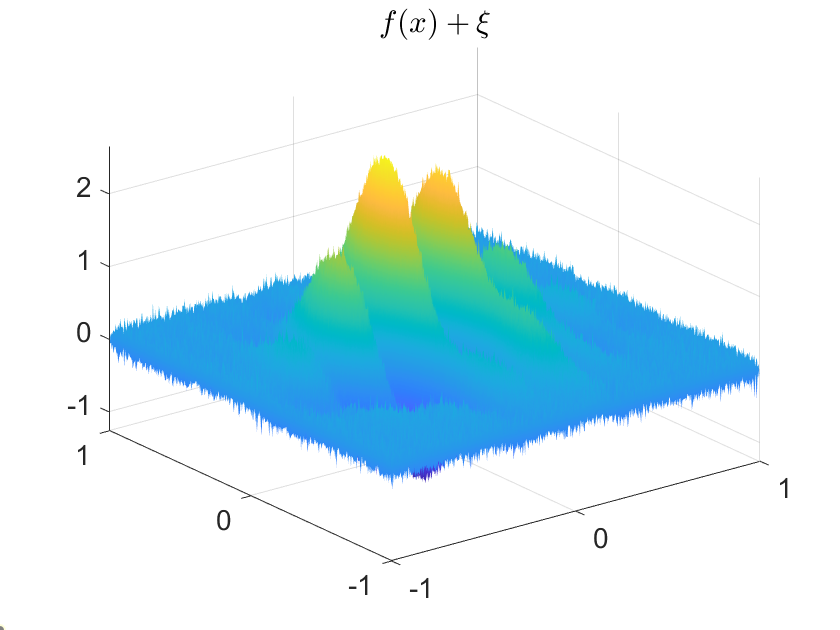}
    \end{minipage}
    \begin{minipage}{.45\textwidth}
    \centering
    \includegraphics[trim=40 0 0 0,clip, height=0.85\textwidth,width=1.1\textwidth]{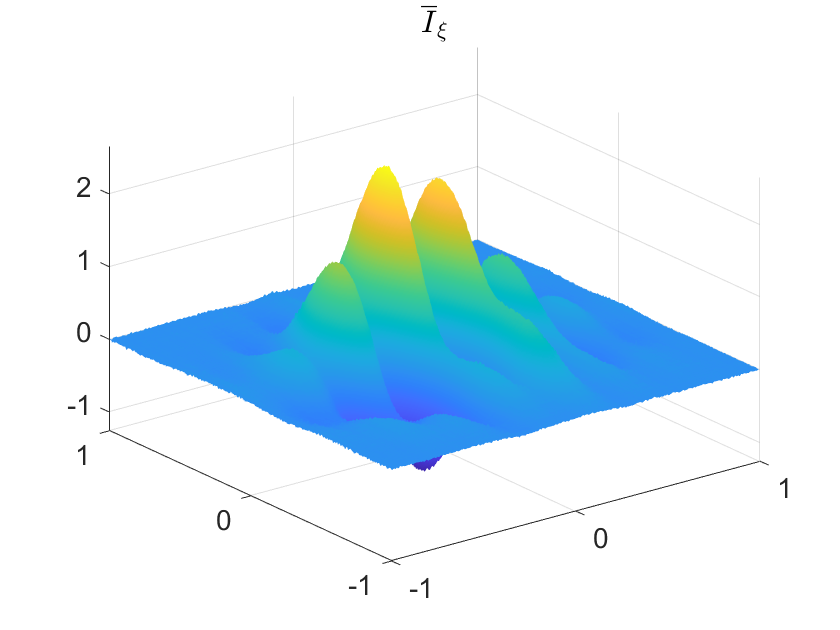}
    \end{minipage}
    \begin{minipage}{.45\textwidth}
    \centering
    \includegraphics[trim=0 0 40 0,clip, height=0.85\textwidth,width=1.1\textwidth]{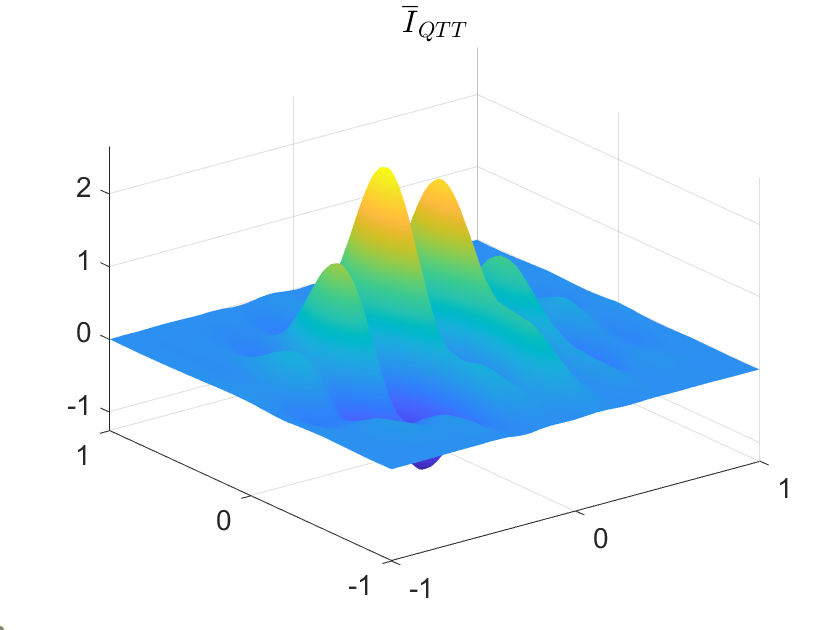}
    \end{minipage}
    \begin{minipage}{.45\textwidth}
    \centering
    \includegraphics[trim=20 0 0 0,clip, height=0.85\textwidth,width=1.1\textwidth]{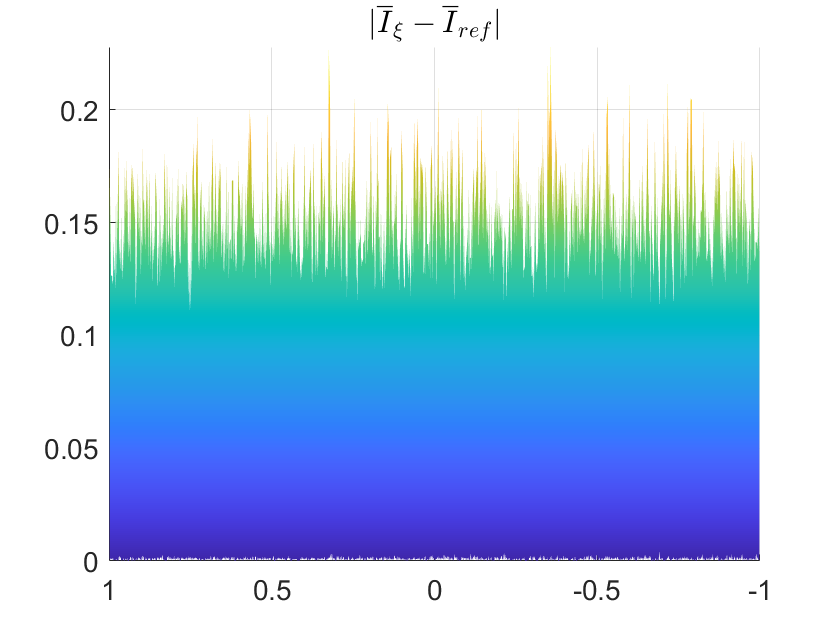}
    \end{minipage}
    \begin{minipage}{.45\textwidth}
    \centering
    \includegraphics[trim=0 0 40 0,clip, height=0.85\textwidth,width=1.1\textwidth]{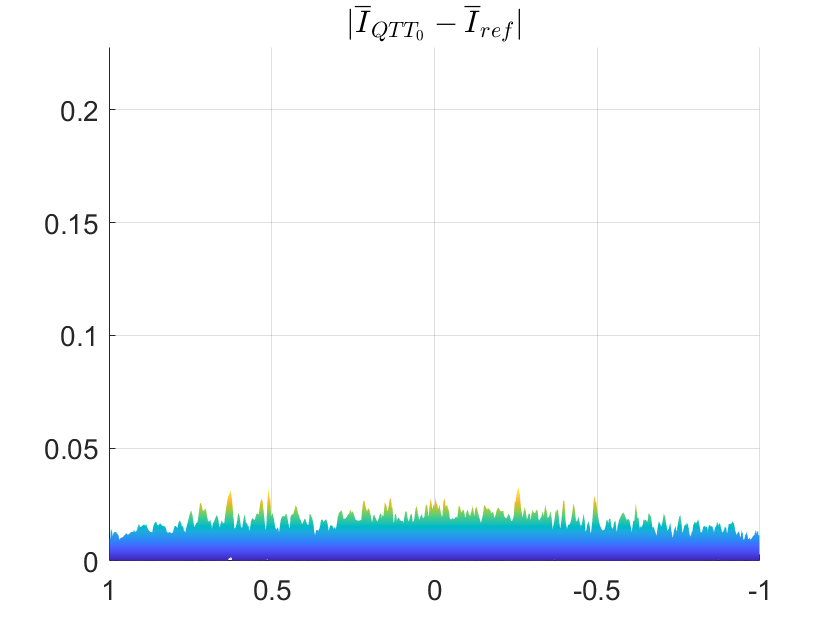}
    \end{minipage}
    \caption{\textit{Top Left:} True convolution of data without noise, $\bbI$. \textit{Top Right:} Function data with noise, $\bbf_\xi$.
    \textit{Middle Left:} True convolution of data with noise, $\bbI_\xi$. \textit{Middle Right:} Convolution using the \textbf{max rank TT-SVD} algorithm, $\bbI_{QTT_0}$.
    \textit{Bottom Left:} Absolute error of $\bbI_\xi$. \textit{Bottom Right:} Absolute error of $\bbI_{QTT_0}$.}
    \label{fig:conv_e3}
\end{figure}

In Table \ref{tab:qtt_rt_e3}, we compare the run times for the different methods of computing the convolution in two spatial dimensions. We get similar results as the one-dimensional case, where the fastest run time is from the convolution with FFT, but with the \textbf{max rank TT-SVD} and \textbf{max rank TT-RSVD} methods approaching its run time asymptotically. In 2D, when there is the same amount of data as in the 1D case (for example, $2^{14\times14}$ in 2D compared to $2^{28}$ in 1D), the 2D examples do not run as fast as the 1D example. This is due to the extra work in the 2D QTT-FFT algorithm from \cite{DKS12}.

\begin{table}[ht!]
    \centering
    \caption{Run times (seconds): Example 3 convolutions.}
    \begin{tabular}{|c|c|c|c|c|c|c|c|}
    \hline
    K   &  $\bbI_\xi$   & $\bbI_{QTT_0}$ & $\bbI_{QTT_r}$ & $\bbI_{\delta}$ & $\bbI_{RTT}$ &  $\bbI_{lr}$ & $\bbI_{QTT_0}/\bbI_\xi$\\
    \hline
    \hline
     $8$ & $\textbf{0.0034}$ & $0.2213$  & $0.310$ & $7.102_{\delta=0.04}$ & $0.297$ & $0.0090_{rank=2}$ & 65.088\\ \hline
     $10$ & $\textbf{0.0629}$  & $0.9209$ & $1.428$ & $0.670_{\delta=0.04}$ & $1.321$ & $0.154_{rank=2}$ & 14.651\\ \hline
     $12$ & $\textbf{0.948}$  & $8.6462$ & $11.258$ & $22.79_{\delta=0.04}$ & $10.27$ & $5.15_{rank=2}$ & 9.121\\ \hline
     $14$ & $\textbf{58.67}$  & $147.8$ & $151.05$ & $1087_{\delta=0.04}$ & $164.5$ & $286.2_{rank=2}$ & 2.519\\ \hline
    \end{tabular}
    \label{tab:qtt_rt_e3}
\end{table}

Again we compare the amount of data stored in the full format versus in the QTT-format. Note that the spatial dimension of the original function does not matter in how much storage it takes, just the dimensionality of the data. For example, it takes just as much data to store a vector in $\R^{2^{20}}$ as it does to store a matrix in $\R^{2^{10}\times2^{10}}$ in the QTT format with a max rank of $R_{max}$.

\begin{table}[ht!]
    \centering
    \caption{Data storage for Example 3.}
    \begin{tabular}{|c|c|c|}
    \hline
    K   &  $\bbf_\xi$   & $\tF_\xi$\\
    \hline
    \hline
     $8$ & $65{,}536$ & $2088$ \\ \hline
     $10$ & $1{,}048{,}576$ & $2888$ \\ \hline
     $12$ & $16{,}777{,}216$ & $3688$ \\ \hline
     $14$ & $268{,}435{,}456$ & $4488$ \\ \hline
    \end{tabular}
    \label{tab:data_storage_e3}
\end{table}

\section{Conclusions}
In this paper, we have shown that the QTT decomposition, along with the QTT-FFT algorithm, can effectively remove noise from signals with full TT-ranks when the true signal is of low rank. As we have seen in the numerical examples, we could drastically remove the amount of noise from the signal compared to if we took the convolution in the traditional way of using the FFT algorithm. This comes at the cost of run time, but our methods still run at a reasonable speed which got closer to the FFT run time as the dimensionality of the data increased. This is demonstrated by three different examples, two in one spatial dimension and one in two spatial dimensions. We are even able to show that our method works on very oscillatory data where it is required to have a sinc kernel with a narrow main lobe. Using approximate TT-ranks smaller than the TT-ranks of the actual signal data,
we are able to recover most of the signal. This indicates that as long as the signal is reasonably smooth, the QTT decomposition can effectively be used for noise reduction for high-dimensional data, even if the true TT-rank is unknown.

From our three new approaches, the \textbf{max rank TT-SVD} convolution algorithm works much better than the \textbf{max rank TT-RSVD} and the \textbf{SV drop off TT-SVD} convolution algorithms. As was stated before, the \textbf{max rank TT-RSVD} can outperform the \textbf{max rank TT-SVD} algorithm when there are larger mode sizes that are used in this paper. For this reason, we present this algorithm, as we have not seen it in the literature elsewhere. The \textbf{SV drop off TT-SVD} convolution algorithms do not produce as accurate of a method as the \textbf{max rank TT-SVD} or the \textbf{max rank TT-RSVD} algorithm; however, in some cases, it does run faster, and this method may give a higher degree of confidence that the truncated singular values are of little importance. Unfortunately, this method can also lead to long run times, as is seen in Table \ref{tab:qtt_rt_e3} when $K=14$.

\appendix

\section{Randomized SVD}
\label{sec:randomized}

Here, we give a brief overview of the randomized SVD (RSVD) decomposition from \cite{HMT11}. To compute the RSVD of the matrix $A\in\R^{m\times n}$, the first step is to find $Q\in\R^{m\times (k+p)}$ such that
\begin{equation*}
    \matA\approx \matQ\matQ^*\matA
\end{equation*}
where $\matQ$ has orthonormal columns and whose columns are approximations for the range of $\matA$. Here, $k$ is the number of singular values that we want in our approximation to be close to the singular values of $\matA$, and $p$ is what is known as an oversampling parameter.
To find $\matQ$, we use the following \aref{alg:Q}.

\begin{algorithm}[H]
    \SetKwInOut{Input}{input}
    \SetKwInOut{Output}{output\,}
    \SetAlgoLined
    \Input{$\matA$, $k$, $p$}
    \Output{$\matQ$}
     Draw random matrix $\matOm\in\R^{n\times(k+p)}$ such that $\matOm_{i,j}\sim \mathcal{N}(0,1)$.\\
     Let $\matY=\matA\matOm$.\\
     Compute QR factorization $\matQ\matR=\matY$.
     \caption{Solving the Fixed-Rank Problem}
     \label{alg:Q}
\end{algorithm}

Once we have obtained $\matQ$, we can compute the low-rank RSVD using \aref{alg:rsvd} (Algorithm 5.1 in \cite{HMT11}).

\begin{algorithm}[H]
    \SetKwInOut{Input}{input}
    \SetKwInOut{Output}{output\,}
    \SetAlgoLined
    \Input{$\matA$, $\matQ$, $k$}
    \Output{$\matU\matSi \matV^*$}
     \begin{enumerate}
         \item Let $\matB=\matQ^*\matA$.
         \item Compute SVD: $\Tilde{\matU}\matSi \matV^*=\matB$.
         \item Let $\matU=\matQ\Tilde{\matU}$.
     \end{enumerate}
     \caption{RSVD}
     \label{alg:rsvd}
\end{algorithm}

With these algorithms, we obtain an approximation $\tilde{\matA}$ to $\matA$ such that
\begin{equation}\label{rsvd_error}
    \|\matA-\tilde{\matA}\|\leq (1+11\sqrt{k+p}\sqrt{\min(m,n)})\sigma_{k+1},
\end{equation}
with probability $1-6p^{-p}$. If we truncate the SVD to only the leading k singular values in \aref{alg:rsvd}, then the error on the left-hand side of \eqref{rsvd_error} only increases by at most $\sigma_{k+1}$.
The computational complexity for each step of this algorithm is given as
\begin{itemize}
    \item $\mathcal{O}(mn(k+p))$
    \item $\mathcal{O}((k+p)^2n)$
    \item $\mathcal{O}((k+p)^2m)$,
\end{itemize}
Thus, for $k+p<\min(m,n)$, the overall algorithm requires $\mathcal{O}(mn(k+p))$ operations.

\subsection*{Acknowledgments} 

The work of A. Chertock was supported in part by NSF grants DMS-1818684 and DMS-2208438. The work of C. Leonard was supported in part by NSF grant DMS-1818684. The work of S. Tsynkov was supported in part by  US Air Force Office of Scientific Research (AFOSR) under grant \# FA9550-21-1-0086.

\bibliographystyle{plain}
\bibliography{ref}

\end{document}